\documentclass[reqno,11pt,a4paper]{amsart}
\usepackage[a4paper,left=30mm,right=30mm,top=30mm,bottom=30mm,marginpar=20mm]{geometry}

\usepackage{amsmath,amsthm,amssymb,amsfonts,mathrsfs,color,hyperref,xcolor}
\usepackage{graphicx}
\usepackage{bbm}
\usepackage{amsbsy}
\usepackage{latexsym}
\usepackage{faktor}
\usepackage[utf8]{inputenc}
\usepackage{cite}
\usepackage{enumerate}
\usepackage[shortlabels]{enumitem}
\usepackage{setspace}
\usepackage{scalerel,stackengine}

\allowdisplaybreaks

\theoremstyle{plain}
\begingroup
\newtheorem{thm}{Theorem}[section]
\newtheorem{lem}[thm]{Lemma}
\newtheorem{prop}[thm]{Proposition}

\endgroup

\theoremstyle{definition}
\begingroup
\newtheorem{defn}[thm]{Definition}
\newtheorem{rem}[thm]{Remark}

\endgroup

\numberwithin{equation}{section}

\newcommand{\res}{\mathop{\hbox{\vrule height 7pt width .5pt depth 0pt
\vrule height .5pt width 6pt depth 0pt}}\nolimits}

\newcommand{\N}{\mathbb N} 
\newcommand{\R}{\mathbb R} 

\newcommand{\wto}{\rightharpoonup}
\newcommand{\wsto}{\stackrel{*}{\rightharpoonup}}
\newcommand{\e}{\varepsilon}

\newcommand{\LL}{{\mathcal L}}
\newcommand{\HH}{{\mathcal H}}
\newcommand{\M}{{\mathcal M}}
\newcommand{\A}{{\Acal}}

\newcommand{\Crm}{\mathrm{C}}

\newcommand{\Hrm}{\mathrm{H}}

\newcommand{\Lrm}{\mathrm{L}}

\newcommand{\Acal}{\mathcal{A}}

\newcommand{\Dcal}{\mathcal{D}}

\newcommand{\Lcal}{\mathcal{L}}

\newcommand{\Rcal}{\mathcal{R}}

\newcommand{\Cbb}{\mathbb{C}}

\DeclareMathOperator{\supp}{supp}

\newcommand{\ii}{\mathrm{i}}

\newcommand{\setb}[2]{\bigl\{\, #1 \ \ \textup{\textbf{:}}\ \ #2 \,\bigr\}}
\newcommand{\setB}[2]{\Bigl\{\, #1 \ \ \textup{\textbf{:}}\ \ #2 \,\Bigr\}}
\newcommand{\setBB}[2]{\biggl\{\, #1 \ \ \textup{\textbf{:}}\ \ #2 \,\biggr\}}

\newcommand{\norm}[1]{\|#1\|}

\newcommand{\altnorm}[1]{{\left\vert\kern-0.25ex\left\vert\kern-0.25ex\left\vert #1 \right\vert\kern-0.25ex\right\vert\kern-0.25ex\right\vert}}

\newcommand{\dprb}[1]{\bigl\langle #1 \bigr\rangle}

\newcommand{\dprBB}[1]{\biggl\langle #1 \biggr\rangle}

\newcommand{\di}{\mathrm{d}}
\newcommand{\dd}{\;\mathrm{d}}

\newcommand{\loc}{\mathrm{loc}}

\newcommand{\per}{\mathrm{per}}

\newcommand{\todown}{\downarrow}

\def\XXint#1#2#3{{\setbox0=\hbox{$#1{#2#3}{\int}$} 
\vcenter{\hbox{$#2#3$}}\kern-.5\wd0}}

\DeclareMathOperator{\tr}{tr}

\DeclareMathOperator{\dive}{div}

\newcommand{\Msn}{\mathbb{M}^{n{\times}n}_{\mathrm{sym}}}

\begin{document}
\title[Shape optimization of light structures]{Shape optimization of light structures with general Hooke's laws}
\author[D. Andreakis]{Dimitrios Andreakis}
\address[D. Andreakis]{Mathematics Institute, University of Warwick, Coventry CV4 7AL, UK.}
\email{Dimitrios.Andreakis@warwick.ac.uk}

\author[F. Iurlano]{Flaviana Iurlano}
\address[F. Iurlano]{Dipartimento di Matematica, Università di Genova, via Dodecaneso 35, 16146 Genoa, Italy}
\email{flaviana.iurlano@unige.it}

\author[F. Rindler]{Filip Rindler}
\address[F. Rindler]{Mathematics Institute, University of Warwick, Coventry CV4 7AL, UK.}
\email{F.Rindler@warwick.ac.uk}
	





\begin{abstract}
{For the minimum-compliance shape optimization problem it was established recently [\textit{Duke Math.\ J.} 172 (2023), 43--103] that a sequence of approximately optimal shapes of exact volume \(\e\) converges to a limiting generalized shape, which is represented by a possibly diffuse probability measure, as \(\e \todown 0\). This generalized shape minimizes the so-called light-structures compliance functional, which --- as conjectured by Bouchitt\'e --- is expressed in terms of a relaxed integrand, reflecting that optimal shapes must only involve ``lower-dimensional'' structures, like the Michell-type trusses that are well known in the engineering domain. However, the aforementioned result only applies for the simplest isotropic quadratic energy density, that is, the squared Frobenius norm, in dimensions two and three. The present work extends the relaxation result to the entire class of (isotropic or anisotropic) elasticity tensors and to any dimension, thus solving Bouchitt\'e's conjecture in full generality. Since the previous work relied on explicit formulas for the relaxed integrands (due to Allaire, Kohn, and Strang), which are not available in the general case, several new arguments and substantial modifications of the original approach are required. We also significantly simplify and streamline other parts of the original argument, in addition to clarifying some related statements in the literature.}
\medskip


\medskip

\noindent\textsc{Date:} \today
\end{abstract}

\maketitle


\section{Introduction}

{The Optimal Light Structures Conjecture~\cite{Bouchitte03}, due to G.~Bouchitt\'{e}, concerns the following optimal design problem: Given an amount of elastic material, one is tasked with distributing this material in the most rigid way possible, as measured by minimal (linearly-)elastic compliance. An \emph{optimal light structure} then is the (rescaled) limit of those optimally rigid distributions of material as the amount of material tends to zero. In~\cite{Bouchitte03} a formula is proposed for the limit functional representing the relaxed compliance. The conjectured shape of the proposed integrand of this functional expresses the fact that only lower-dimensional structures should be used for the optimal solution. This is closely related to the theory of Michell trusses~\cite{Michell04,Olsen00,Olbermann17,Olbermann22}, which says that optimal structures should be made up of beams (trusses) and not truly three-dimensional structures.

In the case of the simplest elastic energy density $\frac12 |\cdot|^2$ (with $|\cdot|$ the Frobenius norm), the Optimal Light Structures Conjecture was recently solved in~\cite{BabadjianIurlanoRindler23} via a strategy inspired by the work~\cite{BabadjianIurlanoRindler21} on the convergence of a damage model to a plasticity model in the diffuse-concentration limit, as well as by insights from~\cite{KohnStrang86_1,KohnStrang86_2,KohnStrang86_3,AllaireKohn93,BouchitteButtazzo01}. However, even for the case of isotropic energy densities, this is not a complete solution since, for instance, it does not cover the classical case of isotropic elasticity with arbitrary Lam\'e constants.

This work completes the analysis and proves the conjecture for all integrands of linearized elasticity.

In the following we describe the setup and our main theorem. A more extensive introduction to the problem, as well as further background and applications, can be found in~\cite{BabadjianIurlanoRindler23}.} Let $n\in\N$, $n\geq2$, and let $\Omega \subset \R^n$ be a bounded $\Crm^2$-domain, namely $\Omega$ is assumed to be open, bounded, connected, and to have boundary $\partial \Omega$ of class $\Crm^2$. Let $\Cbb$ be a fourth-order symmetric elliptic tensor, with entries {$\mathrm{C}_{ijkl}$, $i,j,k,l=1,\dots,n$}. Let $j$ be the quadratic form associated to $\Cbb$,
\[
  j(\xi) := \frac12 \dprb{\Cbb\xi  ,\xi}  \qquad\text{for all symmetric matrices $\xi \in \Msn$.}
\]
We recall the definition of the Legendre--Fenchel convex conjugate of $j$,
  \begin{equation}\label{Leg-Fen}
  j^*(\tau) := \sup_{\xi \in \Msn} \bigl\{ \dprb{\xi,\tau} - j(\xi) \bigr\},  \qquad \tau \in \Msn,
\end{equation} and also introduce the function (a kind of ``partial'' bi-conjugate)
\begin{equation}\label{Leg-Fen-wave}\bar j(\xi):=\sup_{\tau \in \Lambda_{\dive}}\left\{\dprb{\xi , \tau} - j^*(\tau) \right\}, \qquad \xi \in \Msn,
\end{equation}
where $\Lambda_{\dive}:=\{\tau\in \Msn:\ \det\tau=0\}$ denotes the {Tartar wave cone~\cite{Tartar79,Rindler26book}} of the divergence operator.

We fix $f\in\Lrm^2(\partial\Omega;\R^n)$. It will be regarded as an element of $\Hrm^{-1}(\R^n;\R^n)$ via  $$f(v):=  \int_{\partial\Omega}f\cdot \ v \dd\HH^{n-1}  \qquad \text{for} \ v\in \Hrm^{1}(\R^n;\R^n),$$  
or as an element of $\M(\R^n;\R^n)$ via $$ f(A):=\int_{A\cap\partial\Omega}f \dd\HH^{n-1} \qquad \text{for }A\subset \R^n \text{ Borel.}$$

We also require that
\[
  \int_{\partial\Omega} f\cdot r \dd\HH^{n-1}=0  \qquad\text{for all rigid motions $r \in \Rcal$,}
\]
with $\Rcal$ containing all skew-affine maps, see~\eqref{eq:rigid} below.

For $0 < \e < \LL^n(\Omega)$, the set of admissible shapes is
 $$ \A_\e := \setb{ \omega \subset \Omega }{ \text{$\omega$ Lipschitz domain, $\partial\Omega\subset\partial\omega$, and $\LL^n(\omega) = \e$} }.$$
 We also set 
$$X_\e(\Omega) :=\setBB{(\sigma,\mu) \in X(\Omega) }{ \mu=\frac{\LL^n\res \omega}{\e} \text{ for some }\omega \in \A_\e }$$
and
$$X(\Omega) :=\setb{(\sigma,\mu) }{ \mu \in \M^1(\overline\Omega) \text{ and } \sigma \in \Lrm^2(\R^n,\mu;\Msn) },$$
where $\M^1(\overline\Omega)$ denotes the space of probability measures supported in $\overline\Omega$.
We then define the energy functionals $\mathscr E_\e \colon X(\Omega) \to [0,\infty]$ and $\overline{\mathscr E} \colon X(\Omega) \to [0,\infty)$, for $(\sigma,\mu) \in X(\Omega)$, as
\[
  \mathscr E_\e(\sigma,\mu) := \begin{cases}
    \displaystyle \int_{\R^n} j^*(\sigma) \dd \mu   & \text{ if } (\sigma,\mu) \in X_\e(\Omega),\\
    +\infty  &\text{ otherwise,}
   \end{cases}
\]
and
  $$\overline{\mathscr E}(\sigma,\mu) := \int_{\R^n} \bar j^*(\sigma) \dd \mu.\quad $$
   We finally introduce the compliances corresponding to the energies $\mathscr E_\e, \overline{\mathscr E}$, respectively, as \begin{equation}\label{eq:Ceps}
\mathscr C_\e(\mu) :=  \inf \setB{ \mathscr E_\e(\sigma,\mu) }{ \sigma \in \Lrm^2(\R^n,\mu;\Msn),\; -\dive (\sigma\mu)=f  \text{ in }\Dcal'(\R^n;\R^n)}
\end{equation}
and
\begin{equation}\label{eq:barC}
\overline{\mathscr C}(\mu) :=  \inf \setB{ \overline{\mathscr E}(\sigma,\mu) }{ \sigma \in \Lrm^2(\R^n,\mu;\Msn),\; -\dive (\sigma\mu)=f  \text{ in }\Dcal'(\R^n;\R^n)}
\end{equation}
for $\mu \in \M^1(\overline\Omega)$.

The main result of this work is the following:\newpage

\begin{thm} \label{thm:conv-min} \phantom{.}
\begin{enumerate}[(i)]
\item {\bf Convergence of almost-minimizers (lower bound).} For $\e>0$, let $\omega_\e \in \A_\e$ be such that 
\[
  \mathscr C_\e \biggl( \frac{\LL^n\res \omega_\e}{\e} \biggr) \leq \inf_{\omega\in\A_\e} \mathscr C_\e \biggl( \frac{\LL^n\res \omega}{\e} \biggr)  + \alpha_\e,  \qquad\text{where}\qquad \alpha_\e \todown 0.
\]
Then, there exists a sequence $\{\e_k\}_{k \in \N}$ with $\e_k \todown 0$ such that the probability measures
\[
  \frac{1}{\e_k}\LL^n\res\omega_{\e_k},  \qquad k \in \N,
\]
converge weakly* in $\M(\R^n)$ to a solution $\bar\mu \in \M^1(\overline\Omega)$ of
\begin{equation}\label{eq:mass-opt}
\min_{\M^1(\overline\Omega)} \overline{\mathscr C},
\end{equation} 
and
\[
  \lim_{k \to \infty} \mathscr C_{\e_k} \biggl( \frac{\LL^n\res \omega_{\e_k}}{\e_k} \biggr)
  = \overline{\mathscr C}(\bar \mu)=\min_{\M^1(\overline\Omega)} \overline{\mathscr C}.
\]

\item {\bf  Convergence of minimum values (upper bound).}  Let $\mu^* \in \M^1(\overline\Omega)$ be a solution of~\eqref{eq:mass-opt}. Then, for $\e>0$, there exists $\omega^*_\e \in \A_\e$ such that 
$$\frac{\LL^n\res\omega^*_\e}{\e} \wsto \mu^*\quad\text{ in }\M(\R^n)$$
as $\e\to 0$, and
$$\min_{\M^1(\overline\Omega)} \overline{\mathscr C}= \overline{\mathscr C}(\mu^*)=\lim_{\e \to 0}\mathscr C_\e\left(\frac{1}{\e}\LL^n\res\omega^*_\e\right)= \lim_{\e \to 0} \inf_{\omega \in \A_\e} \mathscr C_\e \biggl( \frac{\LL^n\res \omega}{\e} \biggr).$$
\end{enumerate}
\end{thm}
As mentioned before, the above result was proved in the special case $j=\frac12|\cdot|^2$ and for $n=2,3$ in~\cite{BabadjianIurlanoRindler23}. {In those cases, \emph{explicit} formulas are known for the relaxation of the so-called Kohn--Strang integrand~\cite{KohnStrang86_1,KohnStrang86_2,KohnStrang86_3,AllaireKohn93,AllaireKohn93b,Allaire02book} and this simplifies the analysis considerably. In the general case, only the abstract Hashin–Shtrikman bounds~\cite{AllaireKohn93,Allaire02book} may be employed. Additionally, the compensated compactness and convex analysis arguments of~\cite{BabadjianIurlanoRindler23} relied on involved linear algebra estimates between singular values and eigenvalues, which only work for the case of the Frobenius norm. It turns out, however, that a different way of reasoning not only generalizes readily to arbitrary quadratic forms, but also provides a cleaner argument in the cases treated in~\cite{BabadjianIurlanoRindler23}. Finally, we also need to clarify some arguments in the literature concerning $\dive$-quasiconvex hulls.}

\subsection*{Acknowledgements}
Dimitrios Andreakis is supported by the Warwick Mathematics Institute Centre for Doctoral Training, and gratefully acknowledges funding from the University of Warwick and the UK Engineering and Physical Sciences Research Council (Grant number: EP/W524645/1). Flaviana Iurlano is member of GNAMPA of INdAM. Filip Rindler gratefully acknowledges funding from the UKRI Frontier Research Guarantee (ERC guarantee) grant EP/Z000297/1 (ERC CONCENTRATE).

\section{Preliminaries}

\subsection{Notation}
We denote by $\mathbb M^{n \times n}$ the space of $n\times n$ matrices, and by $\Msn$ and $\mathbb M^{n \times n}_{\rm skew}$ the subspaces of symmetric and skew-symmetric matrices, respectively. If $\xi, \zeta \in \mathbb M^{n \times n}$, then $\dprb{\xi,\zeta}:=\tr(\xi^T \zeta)$ denotes their Frobenius product, and $|\xi|:=\sqrt{\dprb{\xi,\xi}}$ is the corresponding Frobenius norm. We denote the Lebesgue measure in $\R^n$ by $\Lcal^n$ and the $k$-dimensional Hausdorff measure by $\HH^k$.

Let $Q=(0,1)^n$ be the open unit cube in $\R^n$. We say that $\phi \colon \R^n \to \Msn$ is $Q$-periodic if $\phi(x+e_i)=\phi(x)$ for every $x\in\R^n$ and every $i=1,\dots,n$. For measurable maps, this equality is understood to hold for $\Lcal^n$-a.e. $x\in\R^n$. We denote by $\Crm^\infty_{\rm per}(Q;\Msn)$ the space of all $Q$-periodic functions in $\Crm^\infty(\R^n;\Msn)$. The spaces $\Lrm^2_{\rm per}(Q;\Msn)$ and $\Hrm^1_{\rm per}(Q;\R^n)$ are defined analogously.

We denote by $\M(\R^n;\Msn)$, respectively by $\M(\R^n)$, the space of $\Msn$-valued, respectively real-valued, finite Radon measures on $\R^n$. The sets $\M^+(\R^n)$ and $\M^1(\R^n)$ denote the spaces of nonnegative {locally finite} Radon measures and probability measures on $\R^n$, respectively. If $K \subset \R^n$ is compact, we denote by $\M(K;\Msn)$ the space of $\Msn$-valued finite Radon measures on $\R^n$ with support contained in $K$. Corresponding definitions are used for $\M(K)$, $\M^+(K)$, and $\M^1(K)$. If $f\in \Lrm^1(\R^n)$, we will often identify $f$ with the absolutely continuous measure $f\Lcal^n$.

Given two nonnegative quantities $a$ and $b$, we write $a\sim b$ if there exist constants $c,C>0$, independent of the relevant parameters, such that $ca\leq b\leq Ca$. Throughout this work, $c$ and $C$ denote positive constants that may change from line to line.

The {{\it Tartar wave cone}~\cite{Tartar79} (also see Section~8.2 in~\cite{Rindler26book} for its relation to the theory of compensated compactness)} associated with the row-wise divergence ``$\dive$'' is the set $\Lambda_{\dive} \subset \Msn$ given as
$$
\Lambda_{\dive}:=\bigcup_{\lambda\neq0}\ker\ \mathbb K(\lambda),
$$
where
$$
 \mathbb K(\lambda)M:=(2\pi\ii)M\lambda,
$$
for $M\in\Msn$ and $\lambda\in\R^n$. Equivalently,
$$
\Lambda_{\dive}=\setb{M\in\Msn}{\det M=0}.
$$

For every subspace $V$ of $\Msn$, denote its complexification by
$$
V_{\mathbb C}:=\setb{x+\ii y}{x,y\in V}.
$$
Moreover, we let $\Pi_V\xi$ be the orthogonal projection of $\xi\in\Msn$ onto $V$.

Finally, we define
\begin{equation} \label{eq:rigid}
\Rcal:=\setB{r\colon\R^n\to\R^n}{r(x)=Sx+b,\ S\in\mathbb M^{n\times n}_{\rm skew},\ b\in\R^n}
\end{equation}
to be the space of rigid body motions.

\subsection{Convex analysis} \label{sc:convex}
{As the fourth-order tensor $\Cbb$ from the statement of the main theorem is assumed to be elliptic,} there are positive constants $c$ and $C$ such that
\begin{equation}
\label{eq:forallsloth}
c|\xi|^2\leq j(\xi)\leq C|\xi|^2 , \qquad \xi \in \Msn.
\end{equation}
Also, 
$$j^*(\tau)=\frac12 \dprb{\Cbb^{-1}\tau,\tau}, \qquad \tau \in \Msn,$$
hence
\begin{equation}\label{eq:alpha}
c|\tau|^2\leq j^*(\tau)\leq C|\tau|^2 , \qquad \tau \in \Msn.
\end{equation}
Recalling the definition in~\eqref{Leg-Fen-wave}, it is easy to see that $\bar j$ is convex, as well as positively $2$-homogeneous.
We also define the \emph{gauge function} $\rho \colon \Msn \to [0,\infty)$ of the convex set $\{\bar j \leq \frac12\}$ by
\begin{equation}\label{eqrho}
\rho(\xi):=\inf\setBB{t>0 }{ \bar j\left(\frac{\xi}{t}\right) \leq \frac12 }, \qquad \xi \in \Msn.
\end{equation}
This {function} is convex, continuous, and positively 1-homogeneous.
Due to the 2-homogeneity of $\bar j$, we get
\[
  \bar j = \frac12 \rho^2.
\]
We also introduce the \emph{polar function} $\rho^\circ \colon \Msn\to [0,\infty)$ of $\rho$, defined by
$$\rho^\circ(\tau):=\sup_{\rho(\xi)\leq 1}\dprb{\xi,\tau}=\sup_{\bar j(\xi)\leq \frac12}\dprb{\xi,\tau}, \qquad \tau \in \Msn.$$
According to~\cite[Corollary~15.3.1]{Rockafellar70book}, for the convex conjugate $\bar j^* \colon \Msn \to [0,\infty)$ of $\bar j$, defined by
\[
  \bar j^*(\tau):=\sup_{\xi\in\Msn} \bigl\{\dprb{\xi,\tau}-\bar j(\xi)\bigr\}, \qquad \tau \in \Msn,
\]
we have
\begin{equation}\label{eq:maxalpha}
  \bar j^*=\frac12 (\rho^\circ)^2.
\end{equation}

\subsection{Compensated compactness} \label{sc:CC}
We shall use the following well-known compensated compactness lemma, which will allow us to modify the integrand $j^*$ by adding and subtracting a suitable quadratic form, whose contribution has the correct sign under the relevant asymptotic differential constraints  (see Proposition~\ref{prop:liminf}). For the proof of the lemma we refer to~\cite[Theorem 11]{Tartar79} (see also~{\cite[Theorem~8.7]{Rindler26book}}).
\begin{lem}
\label{lem:sloth sympageia}
Suppose that $\tau_k \wto 0$ weakly in $\Lrm^2(\R^n;\Msn),$ $\dive\tau_k \to 0$ strongly in $\Hrm^{-1}(\R^n;\R^n)$ and also that the supports of $\tau_k$ are contained in a bounded set. Let $q$ be a quadratic form such that $$q(\tau)\geq 0,  \qquad \tau \in \Lambda_{\dive},$$ and assume that the sequence of functions {$q \circ \tau_k$} is uniformly bounded in $\Lrm^1(\R^n)$. Then, for any $\varphi \in \Crm^\infty_c(\R^n)$ it holds that $$\liminf_{k \to \infty} \int_{\R^n} \varphi^2 q(\tau_k) \dd x \geq 0.$$
\end{lem}

We now define the quadratic forms used in the sequel.
\begin{defn}
For $\xi \in \Msn$ such that $\rho(\xi) \leq 1$ (or, equivalently,  $\bar j(\xi) \leq \frac12$), we define the quadratic form $Q_\xi \colon \Msn \to \R$ by
\begin{equation}\label{eq:Qxi} Q_\xi(\tau):=j^*(\tau) - \frac12 \dprb{\xi,\tau}^2, \qquad \tau \in \Msn.
\end{equation}
\end{defn}
In order to apply the above lemma, we need to show that the quadratic form $Q_\xi$ is nonnegative on $\Lambda_{\dive}$. To prove this fact, we use the following auxiliary result.
\begin{lem} \label{jbar}
For all $\tau \in \Lambda_{\dive}$,
\[
  \bar j^*(\tau) = j^*(\tau).
\]
\end{lem}

\begin{proof}
{Observe first that $j^*=j^{***}$ by the convexity and lower semicontinuity of $j^*$.}

On the one hand, by definition of $\bar j^*$, we have {by standard results from convex analysis that}
\[
\bar j^*=(j^*+\chi_{\Lambda_{\dive}})^{**}\leq j^*+\chi_{\Lambda_{\dive}},
\]
where {$\chi_{\Lambda_{\dive}}$} denotes the indicator function of $\Lambda_{\dive}$, that is, $\chi_{\Lambda_{\dive}}(\tau):=0$ if $\tau\in \Lambda_{\dive}$ and $\chi_{\Lambda_{\dive}}(\tau):=+\infty$ otherwise. Hence,
\[
\bar j^*(\tau) \leq j^*(\tau)
\]
for every $\tau \in \Lambda_{\dive}$.
On the other hand,
\[
j^*\leq j^*+\chi_{\Lambda_{\dive}},
\]
and therefore,
\[
j^*=j^{***}\leq (j^*+\chi_{\Lambda_{\dive}})^{**}=\bar j^*.
\]
Consequently,
\[
\bar j^*(\tau)=j^*(\tau)
\]
for every $\tau\in\Lambda_{\dive}$.
\end{proof}
 By the definition $\bar j=(j^*+\chi_{\Lambda_{\dive}})^{*}$, together with~\eqref{eq:forallsloth} and~\cite[Equation (2.3)]{BabadjianIurlanoRindler23}, we also obtain that there exists $c>0$ such that \begin{equation} \label{eq:jbar_est} c^{-1} |\xi|^2 \leq \bar j(\xi) \leq c|\xi|^2, \qquad c^{-1} |\xi| \leq \rho^\circ(\xi) \leq c|\xi|, \qquad \xi \in \Msn. \end{equation} We are now ready to prove our generalization of~\cite[Lemma 2.4]{BabadjianIurlanoRindler23}.
\begin{lem}\label{Qxi}
For all $\tau \in \Lambda_{\dive}$ and all $\xi\in\Msn$ such that $\bar j(\xi)\leq \frac12$ one has
\[ 
  Q_\xi(\tau)\geq 0.
\]
\end{lem}

\begin{proof}
	Since the set $K:=\{\xi\in\Msn:\bar j(\xi)\leq 1/2\}$ is symmetric with respect to the origin, we have
	\[
	\sup_{\xi\in K} \bigl|\dprb{\xi,\tau}\bigr|
	=
	\sup_{\xi\in K}\dprb{\xi,\tau},
	\]
	so that
	\[
	\sup_{\xi\in K}\dprb{\xi,\tau}^2
	=
	\left(\sup_{\xi\in K}\dprb{\xi,\tau}\right)^2.
	\]
	For every $\tau \in \Msn$, we {thus derive from~\eqref{eq:maxalpha}} that \[ \bar j^*(\tau) =\frac12 \sup_{{\xi\in K}}\dprb{\xi,\tau}^2 =\sup_{{\xi\in K}}\left(j^*(\tau)-Q_\xi(\tau)\right). \]
    Therefore, for every fixed $\xi\in\Msn$ such that $\bar j(\xi)\leq \frac12$, {and every $\tau \in \Msn$,} \[ \bar j^*(\tau)\geq j^*(\tau)-Q_\xi(\tau). \] If now $\tau\in\Lambda_{\dive}$, then, by Lemma~\ref{jbar}, \[ Q_\xi(\tau)\geq j^*(\tau)-\bar j^*(\tau)=0, \] {proving the claim.}
\end{proof}
\subsection{The dual of $\Hrm^1$ modulo rigid deformations}
The following representation result is the same as~\cite[Proposition 2.5]{BabadjianIurlanoRindler23}, except that the final equality is {corrected to} the equivalence relation $\sim$. The proof is otherwise unchanged.
\begin{prop} \label{prop:riesz type theorem}
Let $\Omega\subset\R^n$ be a bounded Lipschitz domain. Let $g \in \Hrm^{-1}(\R^n;\R^n)$ with $\supp(g) \subset \overline\Omega$ and $g(r)=0$ for all $r \in \Rcal$. Then, there exists $G \in \Lrm^2(\Omega;\Msn)$ such that
\[
  \dprb{ g,v }=\int_\Omega \dprb{G,e(v)}\dd x, \qquad  v \in \Hrm^1(\R^n;\R^n)
\]
 and
\[
  \|G\|_{\Lrm^2(\Omega;\Msn)}  \sim  \|g\|_{\Hrm^{-1}(\R^n;\R^n)}.
\]

\end{prop}

\section{Lower bound} \label{sc:lower}

{The fundamental compactness result for our situation is the following:}

\begin{prop} \label{prop:compactness}
Assume that for all $\e > 0$ we are given
\[
  (\sigma_\e,\mu_\e) \in X_\e(\Omega)
\]
such that \[
  \sup_{\e > 0} \int_{\R^n}|\sigma_\e|^2\dd\mu_\e < \infty  \qquad\text{and}\qquad
  \sup_{\e > 0} \, \norm{\dive (\sigma_\e\mu_\e)}_{\Hrm^{-1}(\R^n;\R^n)} < \infty.
\]
Then, there exist a sequence $\{\e_k\}_{k\in \N}$ with $\e_k \todown 0$ and $(\sigma,\mu) \in X(\Omega)$ with $\dive(\sigma\mu) \in \Hrm^{-1}(\R^n;\R^n)$ such that
\[
 \begin{cases}
\mu_{\e_k} \wsto \mu &\text{in $\M(\R^n)$}, \\
\sigma_{\e_k} \mu_{\e_k} \wsto \sigma \mu  &\text{in $\M(\R^n;\Msn)$},\\
\dive(\sigma_{\e_k}\mu_{\e_k})\wto\dive(\sigma \mu) &\text{in }\Hrm^{-1}(\R^n;\R^n),
\end{cases}
\]
and
\[
\begin{cases}
  \sqrt{\e_k} \, \sigma_{\e_k} \mu_{\e_k} \wto 0  & \text{in $\Lrm^2(\R^n;\Msn)$},\\
  \dive( \sqrt{\e_k} \, \sigma_{\e_k} \mu_{\e_k}) \to 0  &\text{in $\Hrm^{-1}(\R^n;\R^n)$.}
  \end{cases}
\]
\end{prop}

\begin{proof}
See ~\cite[Proposition~3.1]{BabadjianIurlanoRindler23}.
\end{proof}

We now prove the lower bound. The proof follows the argument of~\cite{BabadjianIurlanoRindler23}, with the use of the extended convex analysis lemmas proved above.

\begin{prop}\label{prop:liminf}
Let $\mu \in \mathcal M^1(\overline\Omega)$ and $\{\mu_\e\}_{\e>0}$ be a family in $\M^1(\overline\Omega)$ such that $\mu_\e \wsto \mu$ in $\mathcal M(\R^n)$. Then,
$$\overline{\mathscr C}(\mu) \leq \liminf_{\e \todown 0}\mathscr C_\e(\mu_\e).$$
\end{prop}

\begin{proof}
If $\liminf_{\e \todown 0} \mathscr C_\e(\mu_\e)=\infty$, the inequality is trivial. We may therefore assume that this $\liminf$ is finite. Arguing exactly as in~\cite{BabadjianIurlanoRindler23}, we extract a sequence $\e_k\todown0$, set $\mu_k:=\mu_{\e_k}$, and find $\sigma_k\in \Lrm^2(\R^n,\mu_k;\Msn)$ and $\omega_k\in\A_{\e_k}$ such that \[ -\dive(\sigma_k\mu_k)=f \quad\text{in }\Dcal'(\R^n;\R^n), \qquad \mu_k=\frac{\LL^n\res\omega_k}{\e_k}, \] and \[ \liminf_{k\to\infty}\int_\Omega j^*(\sigma_k)\dd\mu_k\leq\lim_{k\to\infty}\mathscr C_{\e_k}(\mu_k) = \liminf_{\e\todown0}\mathscr C_\e(\mu_\e).\]
 Moreover, up to a subsequence, there exists $\sigma\in \Lrm^2(\R^n,\mu;\Msn)$ with \[ -\dive(\sigma\mu)=f \quad\text{in }\Dcal'(\R^n;\R^n), \qquad \sigma_k\mu_k\wsto\sigma\mu \quad\text{in }\M(\R^n;\Msn), \] and, setting $\tau_k:=\sqrt{\e_k}\,\sigma_k\mu_k$, \[ \tau_k\rightharpoonup0 \quad\text{weakly in }\Lrm^2(\R^n;\Msn), \qquad \dive\tau_k\to0 \quad\text{strongly in }\Hrm^{-1}(\R^n;\R^n). \] Let \[ \gamma_k:=j^*(\sigma_k)\mu_k\res\Omega. \] Then, up to a further subsequence, $\gamma_k\wsto\gamma$ in $\M(\R^n)$ for some $\gamma\in\M^+(\overline\Omega)$, and \[ \gamma(\R^n) \leq \liminf_{\e\todown0}\mathscr C_\e(\mu_\e). \] It remains to prove that \[ \frac{\di\gamma}{\di\mu}(x)\geq \bar j^*(\sigma(x)) \quad\text{for $\mu$-a.e. }x\in\overline\Omega. \]
 
 {Fix $x_0$ in the set of $\Lrm^2$-Lebesgue points of $\sigma$ with respect to $\mu$ and such that the Radon--Nikod\'ym derivative 
$\frac{\di \gamma}{\di\mu}(x_0)$ exists and is finite,
 and choose radii $\varrho_j\todown0$ such that $\gamma(\partial B_{\varrho_j}(x_0))=0$. By standard results of measure theory, this is possible for $x_0$ from a set of full $\mu$-measure.} Let $\xi\in\Msn$ with $\rho(\xi)\leq1$ and let $\varphi\in\Crm_c^\infty(B_{\varrho_j}(x_0))$, $0\leq\varphi\leq1$. By Lemma~\ref{Qxi}, \[ Q_\xi(\tau)\geq0 \qquad\text{for every }\tau\in\Lambda_{\dive}. \] Therefore, applying Lemma~\ref{lem:sloth sympageia} to the sequence $\tau_k$ and to the quadratic form $Q_\xi$, we obtain \[ \liminf_{k\to\infty} \int_{B_{\varrho_j}(x_0)}\varphi^2 Q_\xi(\tau_k)\dd x \geq0. \] Since \[ j^*(\tau)=Q_\xi(\tau)+\frac12\dprb{\xi,\tau}^2, \] we infer, exactly as in~\cite{BabadjianIurlanoRindler23}, that \[ \gamma(B_{\varrho_j}(x_0)) \geq \liminf_{k\to\infty} \int_{B_{\varrho_j}(x_0)} \frac{\varphi^2}{2}\dprb{\xi,\sigma_k}^2\dd\mu_k. \] The remaining lower semicontinuity argument is unchanged and gives \[ \gamma(B_{\varrho_j}(x_0)) \geq \int_{B_{\varrho_j}(x_0)} \frac{\varphi^2}{2}\dprb{\xi,\sigma}^2\dd\mu. \] Letting $\varphi\uparrow1$ in $B_{\varrho_j}(x_0)$, dividing by $\mu(B_{\varrho_j}(x_0))$, and then letting $j\to\infty$, we obtain \[ \frac{\di\gamma}{\di\mu}(x_0) \geq \frac12\dprb{\xi,\sigma(x_0)}^2. \] Taking the supremum over all $\xi\in\Msn$ with $\rho(\xi)\leq1$ and using~\eqref{eq:maxalpha}, we get \[ \frac{\di\gamma}{\di\mu}(x_0) \geq \bar j^*(\sigma(x_0)). \] Hence \[ \gamma(\R^n) \geq \int_{\overline\Omega}\bar j^*(\sigma)\dd\mu. \] Combining the previous inequalities, we conclude that \[ \liminf_{\e\todown0}\mathscr C_\e(\mu_\e) \geq \gamma(\R^n) \geq \int_{\overline\Omega}\bar j^*(\sigma)\dd\mu = \overline{\mathscr E}(\sigma,\mu) \geq \overline{\mathscr C}(\mu), \] which proves the lower bound.
\end{proof}

\begin{rem}\label{rem:liminf}
Let $A_\e \subset \Omega$ be Lebesgue measurable sets {and let} $\{\mu_\e=\frac{1}{\e}\LL^n \res A_\e\}_{\e>0}$ be a family in $\M^+(\R^n)$ such that $\mu_\e \wsto \mu$ in $\M(\R^n)$ for some $\mu \in \M^+(\R^n)$. Assume that $\sigma_\e \in \Lrm^2(\R^n,\mu_\e;\Msn)$ is such that $\sigma_\e\mu_\e \wsto \sigma\mu $ in $\M(\R^n;\Msn)$
and $\dive(\sigma_\e\mu_\e) \wto \dive(\sigma\mu)$ in $\Hrm^{-1}(\R^n;\R^n)$ for some $\sigma\in \Lrm^2(\R^n,\mu;\Msn)$. Then, as in~\cite[Remark 3.4]{BabadjianIurlanoRindler23}, one can adapt the above arguments to establish the lower bound
$$\liminf_{\e \todown 0}\int_{\R^n} \ j^*(\sigma_\e)\dd \mu_\e \geq \int_{\R^n}\bar j^*(\sigma)\dd\mu.$$
\end{rem}

\begin{rem}\label{rem:comparison}
The lower-bound argument is closely related to Bouchitt\'{e}'s vanishing mass conjecture. In its original formulation, this conjecture concerns an $\Lrm^2$-type concentration regime, which is the scaling naturally arising in our problem. In~\cite{GennaioliRindler26}, an $\Lrm^1$-version of the vanishing-mass principle is proved. Although the two formulations are conceptually very close, the result of~\cite{GennaioliRindler26} does not cover, in its current form, the lower bound needed here, because of the different scaling and the additional difficulty created by a general right-hand side $f$.
\end{rem}

\section{Relaxation of the general Kohn--Strang functional} \label{sc:relax}

{In this section we will compute a pointwise relaxation of the so-called Kohn--Strang functional, which will form the basis of the subsequent analysis. Unlike in~\cite{BabadjianIurlanoRindler23}, no explicit formula for the relaxation is available and so we need to content ourselves with an implicit relaxation formula.}

\subsection{Hooke's laws.}
If $\mathbb{A}$ is a symmetric elliptic fourth order tensor, then it is invertible, and $\mathbb{A} ^{-1}$, $\mathbb{A} ^{1/2} $ and $\mathbb{A}^{-1/2} $ are also symmetric and elliptic. {We refer to such a tensor $\mathbb A$ as a \emph{Hooke's law}.}

We extend $\mathbb A$ to the complexification of $\Msn$ by setting \[ \mathbb A(\xi+\ii \eta):=\mathbb A\xi+\ii\mathbb A\eta, \qquad \xi,\eta\in\Msn. \] For two Hooke's laws $\mathbb A$ and $\mathbb B$, we write $\mathbb A\leq\mathbb B$ if \[ \dprb{\mathbb A \xi, \xi}\leq\dprb{\mathbb B \xi,\xi} \qquad\text{for every }\xi\in\Msn, \] and $\mathbb A<\mathbb B$ if \begin{equation}\label{anisotita telestwn no 2} \dprb{\mathbb A \xi,\xi}<\dprb{\mathbb B \xi,\xi} \qquad\text{for every }\xi\in\Msn,\ \xi\neq0. \end{equation}

\begin{defn}
\label{defn:vasikotatos orismos}
Let $\mathbb A$ and $\mathbb B$ be two Hooke's laws with $\mathbb A<\mathbb B$, let $\theta\in[0,1]$, let $v$ be a unit vector in $\R^n$, and let $R\in SO(n)$ be such that $Re_{i_0}=v$ for some $i_0\in\{1,\dots,n\}$. We set $Q_{v,R}:=RQ$. 
Let $\chi\colon\R^n\to\{0,1\}$ be a Borel function such that $\chi(y+Re_j)=\chi(y)$ for every $j=1,\dots,n$ and for a.e. $y\in\R^n$, and \[ \int_{Q_{v,R}}\chi(y)\dd y=\theta. \]

We define the Hooke's law $\mathbb H(y)$  by {\[ \mathbb H(y)\alpha:=\chi(y)\mathbb A\alpha+(1-\chi(y))\mathbb B\alpha, \qquad \alpha\in\Msn, \]} whose inverse is \[ \mathbb H(y)^{-1}=\chi(y)\mathbb A^{-1}+(1-\chi(y))\mathbb B^{-1}. \] 
Let $\mathbb L_{\mathbb A,\mathbb B,\theta,v,R,\chi}$ be the complementary Hooke's law associated with $\mathbb A,\mathbb B,\theta,v,R,\chi$, that is,
\begin{equation}\label{L}\dprb{\mathbb{L}_{\mathbb{A},\mathbb{B},\theta,v,R,\chi}(\xi),\xi} := \inf_{\zeta}{\int_{Q_{v,R}}}\dprb{\mathbb{H}(y)^{-1} \zeta(y),\zeta(y)}\dd y,
\end{equation} where $\zeta\in \Lrm^2_\loc(\R^n,\Msn)$ {satisfying the \emph{admissibility conditions} $$\dive\zeta=0 \qquad\text{in the weak sense,}$$ $$\int_{Q_{v,R}}\zeta(y) \dd y=\xi,$$ and $$\zeta(y+Re_{j})=\zeta(y) \qquad\text{for all $j \in \left\{1,2,...,n\right\}$ and almost all $y \in \R^n.$} $$} By a standard density argument, it is enough to take the infimum over smooth periodic admissible fields $\zeta$ satisfying the same divergence-free and average constraints. We will write simply $\mathbb{L}$ when {the parameters are clear from the context.}
\end{defn}

\begin{rem}\label{rem:basic properties} Under the notation of Definition~\ref{defn:vasikotatos orismos}, the $k$-th Fourier coefficient of $\zeta$ for $k \in \mathbb{Z}^{n}$ is $$\widehat{\zeta}(k) =\int_{Q_{v,R}}\zeta(y)e^{-2\pi i  Rk\cdot y} \dd y.$$ Then, $\dive\zeta=0$ in the weak sense if and only if $\widehat{\zeta}(k) Rk =0$ for every $k\in\mathbb{Z}^{n}.$
\end{rem}

For every $k\in\R^n\setminus\{0\}$, we set \[{Z(k)}:=\setb{A\in\Msn}{Ak=0}, \qquad {V(\mathbb B,k)}:=\mathbb B^{-1/2}(Z(k)). \]

\begin{prop} \label{prop:sxesi 6.9}
Let $\mathbb A$ and $\mathbb B$ be two Hooke's laws with $\mathbb A<\mathbb B$, and let $v$, $R$, $\chi$, and $\theta\in[0,1]$ be as above. For $h\in\Msn$, set \[ g_c(h):=\sup_{|k|=1}\left|\Pi_{V(\mathbb B,k)}\mathbb B^{1/2}h\right|^2, \] {where we recall that $\Pi_{V(\mathbb B,k)}$ denotes the orthogonal projection on the subspace $V(\mathbb B,k)$.} Then, for every $\xi\in\Msn$, \[ \dprb{\mathbb L \xi,\xi} \geq \dprb{\mathbb B^{-1} \xi,\xi}+S(\xi), \] where \[ S(\xi):= \theta\sup_{h\in\Msn} \bigl\{ 2\dprb{h,\xi} - \dprb{(\mathbb A^{-1}-\mathbb B^{-1})^{-1}h,h} - (1-\theta)g_c(h) \bigr\}. \]

\end{prop}

\begin{proof}
The proof is given in~\cite[Proposition 2.1]{ AllaireKohn93b }. We sketch it for the sake of completeness. We write $Q_v$ instead of $Q_{v,R}$, and $\zeta$ always denotes an admissible field in the sense of Definition~\ref{defn:vasikotatos orismos}. Set \[{\mathbb D}:=\mathbb A^{-1}-\mathbb B^{-1}. \] Since $\mathbb A<\mathbb B$, the operator $\mathbb D$ is positive definite. Hence, for every $h\in\Msn$ and every $\eta\in\Msn$, {by the Legendre--Fenchel inequality,} \[ \dprb{\mathbb D\eta,\eta} \geq 2\dprb{h,\eta}-\dprb{\mathbb D^{-1}h,h}. \] Using the definition of $\mathbb L$ and applying this inequality with $\eta=\zeta(y)$ on the set $\{\chi=1\}$, we obtain \[ \begin{aligned} \dprb{\mathbb L \xi,\xi} &\geq \inf_{\zeta} \int_{Q_v} \left( 2\chi(y)\dprb{h,\zeta(y)} + \dprb{\mathbb B^{-1}\zeta(y),\zeta(y)} \right)\dd y - \theta\dprb{\mathbb D^{-1}h,h}. \end{aligned} \] We estimate the infimum. Since $\mathbb B^{-1/2}$ is symmetric, Plancherel's formula gives \[ \begin{aligned} &\int_{Q_v} \left( 2\chi(y)\dprb{h,\zeta(y)} + \dprb{\mathbb B^{-1}\zeta(y),\zeta(y)} \right)\dd y \\ &\qquad = \sum_{k\in\mathbb Z^n} \left( 2\operatorname{Re} \left[ \overline{\widehat\chi(k)} \dprb{\mathbb B^{1/2}h,\gamma_k} \right] + |\gamma_k|^2 \right), \end{aligned} \] where \[ \gamma_k:=\mathbb B^{-1/2}\widehat\zeta(k). \] The admissibility conditions imply \[ \gamma_0=\mathbb B^{-1/2}\xi, \qquad \gamma_k\in (V(\mathbb B, Rk))_{\mathbb C} \quad\text{for }k\neq0. \] Conversely, these conditions characterize the admissible Fourier modes. Therefore, the infimum can be computed mode by mode. The mode $k=0$ gives {the contribution \[ 2\theta\dprb{h,\xi} + |\mathbb B^{-1/2}\xi|^2 = 2\theta\dprb{h,\xi} + \dprb{\mathbb B^{-1}\xi,\xi}. \]} For $k\neq0$, the minimum over $(V(\mathbb B, Rk))_{\mathbb C}$ is attained at \[ \gamma_k^* = -\widehat\chi(k)\Pi_{V(\mathbb B, Rk)}\mathbb B^{1/2}h \] and is equal to \[ -|\widehat\chi(k)|^2 \left|\Pi_{V(\mathbb B, Rk)}\mathbb B^{1/2}h\right|^2. \] Thus, \[ \begin{aligned} &\inf_{\zeta} \int_{Q_v} \left( 2\chi(y)\dprb{h,\zeta(y)} + \dprb{\mathbb B^{-1}\zeta(y),\zeta(y)} \right)\dd y \\ &\qquad = \dprb{\mathbb B^{-1}\xi,\xi} + 2\theta\dprb{h,\xi} - \sum_{k\neq0} |\widehat\chi(k)|^2 \left|\Pi_{V(\mathbb B, Rk)}\mathbb B^{1/2}h\right|^2. \end{aligned} \] By the definition of $g_c$, \[ \left|\Pi_{V(\mathbb B, Rk)}\mathbb B^{1/2}h\right|^2 \leq g_c(h) \qquad\text{for every }k\neq0. \] Moreover, since $\chi$ takes only the values $0$ and $1$ and has average $\theta$, Parseval's formula yields \[ \sum_{k\neq0}|\widehat\chi(k)|^2 = \int_{Q_v}\chi^2\dd y-\theta^2 = \theta(1-\theta). \] Consequently, \[ \begin{aligned} \dprb{\mathbb L \xi,\xi} &\geq \dprb{\mathbb B^{-1}\xi,\xi} + 2\theta\dprb{h,\xi} - \theta(1-\theta)g_c(h) - \theta\dprb{\mathbb D^{-1}h,h} \\ &= \dprb{\mathbb B^{-1}\xi,\xi} + \theta \left( 2\dprb{h,\xi} - \dprb{(\mathbb A^{-1}-\mathbb B^{-1})^{-1}h,h} - (1-\theta)g_c(h) \right). \end{aligned} \] Taking the supremum over $h\in\Msn$ gives the claimed inequality.	
\end{proof}
 
Before proceeding, we record some consequences of the proof above that will be used later.
\begin{rem}\label{rem:sxedon H-S} For every $h,\xi\in\Msn$, we have \[ \dprb{(\mathbb L-\mathbb B^{-1})(\xi),\xi} \geq -\theta\dprb{(\mathbb A^{-1}-\mathbb B^{-1})^{-1}h,h} +2\theta\dprb{h,\xi}  -\sum_{k\neq0} |\widehat\chi(k)|^2 \left|\Pi_{V(\mathbb B, Rk)}\mathbb B^{1/2}h\right|^2. \] \end{rem}

\begin{rem}\label{rem:Paralili apodeixi gia otan zeta exartate apo xio} Assume in addition that $\chi$ is one-dimensional in the direction $v$, namely that there exists a Borel function $\tilde\chi\colon\R\to\{0,1\}$ such that \[ \chi(y)=\tilde\chi(v\cdot y) \qquad\text{for a.e. }y\in\R^n. \] Let $\zeta$ be an admissible field in Definition~\ref{defn:vasikotatos orismos} which is also one-dimensional in the direction $v$, namely \[ \zeta(y)=\tilde\zeta(v\cdot y) \qquad\text{for a.e. }y\in\R^n \] for some $\tilde\zeta\colon\R\to\Msn$. We denote by $D_\xi$ the class of such $\zeta$'s. Since $v=Re_{i_0}$, Fubini's theorem shows that the only nonzero Fourier modes of $\chi$ and $\zeta$ are those with $k$ parallel to $e_{i_0}$. Therefore, arguing as in the proof of Proposition~\ref{prop:sxesi 6.9}, for every $h\in\Msn$ we obtain \[ \begin{aligned} &\inf_{\zeta\in D_\xi} \int_{Q_v} \left( 2\chi(y)\dprb{h,\zeta(y)} + \dprb{\mathbb B^{-1}\zeta(y),\zeta(y)} \right)\dd y \\ &\qquad = \dprb{\mathbb B^{-1}\xi,\xi} + 2\theta\dprb{h,\xi} -\theta(1-\theta) \left|\Pi_{V(\mathbb B, v)}\mathbb B^{1/2}h\right|^2. \end{aligned} \] \end{rem}

\begin{rem}\label{rem:sxetika me convexity}
For fixed $\xi\in\Msn$, the function \[ T(h):= 2\dprb{h,\xi} - \dprb{(\mathbb A^{-1}-\mathbb B^{-1})^{-1}h,h} - (1-\theta)g_c(h), \qquad h\in\Msn, \] is strictly concave. Indeed, the first term is affine, the second term is strictly concave since $\mathbb A^{-1}-\mathbb B^{-1}$ is positive definite, and the last term is concave because $g_c$ is the supremum of {strictly} convex quadratic forms.
\end{rem}

\begin{rem}\label{rem: sygrisi}
If $\theta\in(0,1)$, then $\mathbb B>\mathbb L^{-1}$. Indeed, let $\xi\in\Msn$, $\xi\neq0$, and take $h=\lambda\xi$, with \[ 0<\lambda< \frac{2|\xi|^2} {\dprb{(\mathbb A^{-1}-\mathbb B^{-1})^{-1}\xi,\xi} +(1-\theta)g_c(\xi)}. \] Since $g_c$ is positively $2$-homogeneous, we have $T(\lambda\xi)>0$. Hence, by Proposition~\ref{prop:sxesi 6.9}, \[ \dprb{\mathbb L \xi,\xi} > \dprb{\mathbb B^{-1}\xi,\xi} \qquad\text{for every }\xi\in\Msn,\ \xi\neq0. \] Thus $\mathbb L>\mathbb B^{-1}$, or equivalently, \( \mathbb B>\mathbb L^{-1}. \)
\end{rem}

The next lemma gives an explicit formula for $\mathbb L$ in the case where $\chi$ is one-dimensional in the direction $v$.

\begin{lem} \label{lem:1st lemma regarding optimality} Let $\mathbb A$ and $\mathbb B$ be two Hooke's laws with $\mathbb A<\mathbb B$, and let $v$, $R$, $\chi$, and $\theta\in(0,1)$ be as in Definition~\ref{defn:vasikotatos orismos}. Assume in addition that $\chi$ is one-dimensional in the direction $v$, namely that there exists a Borel function $\tilde\chi\colon\R\to\{0,1\}$ such that \[ \chi(y)=\tilde\chi(v\cdot y) \qquad\text{for a.e. }y\in\R^n. \] Then \[ \theta(\mathbb L-\mathbb B^{-1})^{-1} = (\mathbb A^{-1}-\mathbb B^{-1})^{-1} + (1-\theta)f_{\mathbb B}^c(v), \] where $f_{\mathbb B}^c(v)$ is the Hooke's law associated with the quadratic form \[ \dprb{f_{\mathbb B}^c(v)(h),h} = \left|\Pi_{V(\mathbb B, v)}\mathbb B^{1/2}h\right|^2, \qquad h\in\Msn. \] \end{lem} \begin{proof} The proof follows~\cite[Proposition 3.1]{AllaireKohn93b}. We sketch the argument, making a few points more explicit. Set \[ \mathbb D:=\mathbb A^{-1}-\mathbb B^{-1}, \qquad {\mathbb F} :=\mathbb D^{-1}+(1-\theta)f_{\mathbb B}^c(v). \] Since $\mathbb A<\mathbb B$, the operator $\mathbb D$ is positive definite. It is enough to prove that, for every $\xi\in\Msn$, \[ \dprb{(\mathbb L-\mathbb B^{-1})(\xi),\xi} = \theta\dprb{\mathbb F^{-1}\xi,\xi}. \] Equivalently, using the dual representation of the inverse of a positive definite operator, it is enough to prove \[ \dprb{(\mathbb L-\mathbb B^{-1})(\xi),\xi} = \theta\sup_{h\in\Msn} \left\{ 2\dprb{h,\xi} - \dprb{\mathbb D^{-1}h,h} - (1-\theta)\dprb{f_{\mathbb B}^c(v)(h),h} \right\}. \] We first prove the lower bound. By Remark~\ref{rem:sxedon H-S}, for every $h\in\Msn$, \[ \dprb{(\mathbb L-\mathbb B^{-1})(\xi),\xi} \geq -\theta\dprb{\mathbb D^{-1}h,h} + 2\theta\dprb{h,\xi}  - \sum_{k\neq0} |\widehat\chi(k)|^2 \left|\Pi_{V(\mathbb B, Rk)}\mathbb B^{1/2}h\right|^2 .  \] Since $\chi$ is one-dimensional in the direction $v$, and since $v=Re_{i_0}$, Fubini's theorem gives $\widehat\chi(k)=0$ unless $k$ is parallel to $e_{i_0}$. For such modes, $V(\mathbb B, Rk)=V(\mathbb B, v)$. Hence, by Parseval's formula, \[ \begin{aligned} \sum_{k\neq0} |\widehat\chi(k)|^2 \left|\Pi_{V(\mathbb B, Rk)}\mathbb B^{1/2}h\right|^2 &= \left|\Pi_{V(\mathbb B, v)}\mathbb B^{1/2}h\right|^2 \sum_{k\neq0}|\widehat\chi(k)|^2 \\ &= \theta(1-\theta) \dprb{f_{\mathbb B}^c(v)(h),h}. \end{aligned} \] Therefore, \[ \begin{aligned} \dprb{(\mathbb L-\mathbb B^{-1})(\xi),\xi} &\geq \theta \left\{ 2\dprb{h,\xi} - \dprb{\mathbb D^{-1}h,h} - (1-\theta)\dprb{f_{\mathbb B}^c(v)(h),h} \right\}. \end{aligned} \] Taking the supremum over $h\in\Msn$ gives \[ \dprb{(\mathbb L-\mathbb B^{-1})(\xi),\xi} \geq \theta\dprb{\mathbb F^{-1}\xi,\xi}. \]

We now prove the opposite inequality. Let $C_\xi$ be the class of admissible fields $\zeta$ in the definition of $\mathbb L$ which are constant on each phase $\{\chi=1\}$ and $\{\chi=0\}$. Since we restrict the class of competitors, \[ \dprb{\mathbb L \xi,\xi} \leq \inf_{\zeta\in C_\xi} \int_{Q_{v,R}} \dprb{\mathbb H(y)^{-1}\zeta(y),\zeta(y)} \dd y. \] For $\zeta\in C_\xi$, the dual representation of the quadratic form associated with $\mathbb D$ gives \[ \begin{aligned} &\int_{Q_{v,R}} \dprb{\mathbb H(y)^{-1}\zeta(y),\zeta(y)} \dd y \\ &\quad = \sup_{h\in\Msn} \biggl[ \int_{Q_{v,R}} \left( 2\chi(y)\dprb{h,\zeta(y)} + \dprb{\mathbb B^{-1}\zeta(y),\zeta(y)} \right)\dd y - \theta\dprb{\mathbb D^{-1}h,h} \biggr]. \end{aligned} \] Thus, by the saddle-point principle~\cite{Rockafellar70book,EkelandTemam76book},  \[ \begin{aligned} &\inf_{\zeta\in C_\xi} \int_{Q_{v,R}} \dprb{\mathbb H(y)^{-1}\zeta(y),\zeta(y)} \dd y \\ &\quad = \sup_{h\in\Msn} \biggl[ \inf_{\zeta\in C_\xi} \int_{Q_{v,R}} \left( 2\chi(y)\dprb{h,\zeta(y)} + \dprb{\mathbb B^{-1}\zeta(y),\zeta(y)} \right)\dd y - \theta\dprb{\mathbb D^{-1}h,h} \biggr]. \end{aligned} \] We now compute the inner infimum. Recall the class \(D_\xi\) of one-dimensional admissible fields appearing in Remark~\ref{rem:Paralili apodeixi gia otan zeta exartate apo xio}. We claim that
\[
\begin{aligned}
&\inf_{\zeta\in C_\xi}
\int_{Q_{v,R}}
\left(
2\chi(y)\dprb{h,\zeta(y)}
+
\dprb{\mathbb B^{-1}\zeta(y),\zeta(y)}
\right)\dd y
\\
&\quad =
\inf_{\zeta\in D_\xi}
\int_{Q_{v,R}}
\left(
2\chi(y)\dprb{h,\zeta(y)}
+
\dprb{\mathbb B^{-1}\zeta(y),\zeta(y)}
\right)\dd y .
\end{aligned}
\]
Indeed, \(C_\xi\subset D_\xi\), which immediately gives the inequality ``\(\geq\)''. Conversely, let \(\zeta\in D_\xi\). Then there exists a function \(\widetilde\zeta\colon\mathbb R\to\Msn\) such that
\[
\zeta(y)=\widetilde\zeta(v\cdot y)
\qquad\text{for a.e. } y\in\mathbb R^n .
\]
The condition \(\dive\zeta=0\) implies that
\(\widetilde\zeta(t)v\) is constant. In particular, if we denote by \(\zeta^{(1)}\) and \(\zeta^{(0)}\) the averages of \(\zeta\) on the two phases \(\{\chi=1\}\) and \(\{\chi=0\}\), respectively, then
\[
\zeta^{(1)}v=\zeta^{(0)}v .
\]
Define
\[
\bar\zeta(y):=\chi(y)\zeta^{(1)}+(1-\chi(y))\zeta^{(0)} .
\]
Then \(\bar\zeta\in C_\xi\). 
Moreover, since the integrand
\[
\eta\mapsto 2\chi(y)\dprb{h,\eta}
+
\dprb{\mathbb B^{-1}\eta,\eta}
\]
is convex in \(\eta\) on each phase, Jensen's inequality gives
\[
\begin{aligned}
&\int_{Q_{v,R}\cap\{\chi=1\}}
\left(
2\dprb{h,\zeta(y)}
+
\dprb{\mathbb B^{-1}\zeta(y),\zeta(y)}
\right)\dd y
\\
&\quad\geq
\int_{Q_{v,R}\cap\{\chi=1\}}
\left(
2\dprb{h,\bar\zeta(y)}
+
\dprb{\mathbb B^{-1}\bar\zeta(y),\bar\zeta(y)}
\right)\dd y,
\end{aligned}
\]
and likewise on \(Q_{v,R}\cap\{\chi=0\}\). Therefore the infimum over \(C_\xi\) is not larger than the infimum over \(D_\xi\), proving the opposite inequality.

By Remark~\ref{rem:Paralili apodeixi gia otan zeta exartate apo xio}, we therefore have
\[
\begin{aligned}
&\inf_{\zeta\in C_\xi}
\int_{Q_{v,R}}
\left(
2\chi(y)\dprb{h,\zeta(y)}
+
\dprb{\mathbb B^{-1}\zeta(y),\zeta(y)}
\right)\dd y
\\
&\quad =
\dprb{\mathbb B^{-1}\xi,\xi}
+
2\theta\dprb{h,\xi}
-
\theta(1-\theta)
\dprb{f_{\mathbb B}^c(v)(h),h}.
\end{aligned}
\] \end{proof}

Let $\mathbb{A}, \mathbb{B}$ be two Hooke's laws with $\mathbb A<\mathbb B$, and let $p\in\N$. For every $j=1,\dots,p$, let $v_j$, $R_j$, $\chi_j$, and $\theta_j\in(0,1)$ be as in Lemma~\ref{lem:1st lemma regarding optimality}. We define recursively a family of complementary Hooke's laws as follows. Set \[ \mathbb L_0:=\mathbb A^{-1}. \] Assume that $\mathbb L_{j-1}$ has been defined and that $\mathbb B>\mathbb L_{j-1}^{-1}$. We set \[ \mathbb L_j := \mathbb L_{\mathbb L_{j-1}^{-1},\mathbb B,\theta_j,v_j,R_j,\chi_j}. \] By Lemma~\ref{lem:1st lemma regarding optimality}, we have \begin{equation}\label{recursive definition} \theta_j(\mathbb L_j-\mathbb B^{-1})^{-1} = (\mathbb L_{j-1}-\mathbb B^{-1})^{-1} + (1-\theta_j)f_{\mathbb B}^{c}(v_j). \end{equation} Moreover, by Remark~\ref{rem: sygrisi}, $\mathbb B>\mathbb L_j^{-1}$. Hence the construction is well-defined for every $j=1,\dots,p$. Let us prove that $\mathbb L_p$ satisfies a structure condition analogous to that obtained for $\mathbb L$ in Lemma~\ref{lem:1st lemma regarding optimality}.
\begin{defn} \label{defn: complementary p-laminate} Let $\mathbb A$ and $\mathbb B$ be two Hooke's laws with $\mathbb A<\mathbb B$, and let $p\in\N$. We say that a Hooke's law $\mathbb L_p$ is a complementary $p$-laminate with respect to $\mathbb A$ and $\mathbb B$ if there exist $v_j$, $R_j$, $\chi_j$, and $\theta_j\in(0,1)$, $j=1,\dots,p$, as in Lemma~\ref{lem:1st lemma regarding optimality}, such that $\mathbb L_p$ is obtained after $p$ steps of the recursive construction above. \end{defn}
\begin{rem}\label{rem:adding laminates by +1} Let $\mathbb{A}, \mathbb{B}$ be two Hooke's laws with $\mathbb{A}< \mathbb{B}$ and $j\in\N.$ If $\mathbb{L'}$ is a complementary j-laminate with respect to $\mathbb{A},\mathbb{B}$, and $\mathbb{L''}$ is a complementary 1-laminate with respect to $(\mathbb{L'})^{-1},\mathbb{B}$ then $\mathbb{L''}$ is a complementary (j+1)-laminate with respect to $\mathbb{A},\mathbb{B}.$
\end{rem}

The proofs of the following two propositions are parallel to ~\cite[Proposition 3.2]{AllaireKohn93b} and ~\cite[Theorem 3.5]{AllaireKohn93b} respectively.

\begin{prop} \label{prop:microstructure for many unit vectors} Let $\mathbb A$ and $\mathbb B$ be two Hooke's laws with $\mathbb A<\mathbb B$, let $\theta\in(0,1)$, let $p\in\N$, and let $m_1,\dots,m_p > 0$ be such that $\sum_{i=1}^p m_i=1$. Let $v_1,\dots,v_p$ be unit vectors in $\R^n$. Then there exists a complementary $p$-laminate $\mathbb L_p$ with respect to $\mathbb A$ and $\mathbb B$ such that \[ \theta(\mathbb L_p-\mathbb B^{-1})^{-1} = (\mathbb A^{-1}-\mathbb B^{-1})^{-1} + (1-\theta)\sum_{i=1}^p m_i f_{\mathbb B}^{c}(v_i).\] Moreover, the laminate can be chosen so that the associated recursive volume
fraction of the phase \(\mathbb A\) is \(\theta\). \end{prop} 
 \begin{prop} \label{prop:proposition regarding optimality for H-S} Let $\mathbb A$ and $\mathbb B$ be two Hooke's laws with $\mathbb A<\mathbb B$, let $\theta\in(0,1)$, and let $\xi\in\Msn$. Then there exist $p\in\N$ and a complementary $p$-laminate $\mathbb L_p$ with respect to $\mathbb A$ and $\mathbb B$, as in Proposition~\ref{prop:microstructure for many unit vectors}, such that \[ \dprb{\mathbb L_p(\xi),\xi} = \dprb{\mathbb B^{-1}\xi,\xi} + S(\xi),\] 
 where $S$ has been {defined in Proposition~\ref{prop:sxesi 6.9}}.\end{prop}

\subsection{Convergence }
We recall the notation introduced in Section~\ref{sc:convex} for the quadratic form $j^*$, as well as the subspaces $Z(k)$ and {$V(\mathbb C,k)$} introduced in Remark~\ref{rem:basic properties} {with $\mathbb B := 2\mathbb C$.} Setting also $\mathbb A_0:=2\Cbb$, we define for every $h\in\Msn$,
\[
g_c(h):=
\sup_{|k|=1}
\left|
\Pi_{V(\mathbb C,k)}\mathbb A_0^{1/2}h
\right|^2.
\]
The following {new} result identifies the relation between $j$ and $g_c$. This relation is the key ingredient in our generalization of the results of~\cite[Chapter 4]{BabadjianIurlanoRindler23}.
\begin{prop} \label{prop:kyrio neo apotelesma tou paper} For every $\tau\in\Msn$, \[ \rho^\circ(\tau)=2\sqrt2\,\sqrt{g_c^*(\tau)}. \] Consequently, for every $\theta>0$, \[ \sup_{h\in\Msn}\bigl\{2\dprb{\tau,h}-\theta g_c(h)\bigr\}<+\infty, \] and $g_c^*(\tau)\geq0$, with equality if and only if $\tau=0$. \end{prop} \begin{proof} We first prove that \begin{equation}\label{eq:jbar_gc_identity} \bar j(2h)=g_c(h) \qquad\text{for every }h\in\Msn. \end{equation} Indeed, by the definition of $\bar j$, \[ \bar j(2h) = \sup_{\tau'\in\Lambda_{\dive}} \bigl\{ 2\dprb{h,\tau'}-j^*(\tau') \bigr\}. \] Since \[ \Lambda_{\dive}=\bigcup_{|k|=1}Z(k), \] and as $j^*(\tau')=|\mathbb A_0^{-1/2}\tau'|^2$, we get \[ \bar j(2h) = \sup_{|k|=1}\ \sup_{\tau'\in Z(k)} \bigl\{ 2\dprb{h,\tau'}-|\mathbb A_0^{-1/2}\tau'|^2 \bigr\}. \]

{Fix $k \in \R^n$ with $|k|=1$. Then $\tau'\in Z(k)$ if and only if $z:=\mathbb A_0^{-1/2}\tau' \in V(\mathbb C,k)$,} and therefore \[ \sup_{\tau'\in Z(k)} \bigl\{ 2\dprb{h,\tau'}-|\mathbb A_0^{-1/2}\tau'|^2 \bigr\} = \sup_{z\in V(\mathbb C,k)} \bigl\{ 2\dprb{\mathbb A_0^{1/2}h,z}-|z|^2 \bigr\} = \left|\Pi_{V(\mathbb C,k)}\mathbb A_0^{1/2}h\right|^2.  \] Taking the supremum over $|k|=1$ gives~\eqref{eq:jbar_gc_identity}. We now take conjugates. From~\eqref{eq:jbar_gc_identity}, \[ \bar j^*(\tau) = \sup_{u\in\Msn} \bigl\{ \dprb{u,\tau}-\bar j(u) \bigr\} = \sup_{h\in\Msn} \bigl\{ 2\dprb{h,\tau}-g_c(h) \bigr\}. \] Since $g_c$ is positively $2$-homogeneous, also $g_c^*$ is positively $2$-homogeneous, and hence \[ \bar j^*(\tau) = g_c^*(2\tau) = 4g_c^*(\tau). \] Using the relation \[ \rho^\circ(\tau)=\sqrt{2\bar j^*(\tau)}, \] {which follows from~\eqref{eq:maxalpha},} we obtain \[ \rho^\circ(\tau) = 2\sqrt2\,\sqrt{g_c^*(\tau)}. \] Finally, for every $\theta>0$, \[ \sup_{h\in\Msn} \bigl\{ 2\dprb{\tau,h}-\theta g_c(h) \bigr\} = \sup_{h'\in\Msn} \biggl\{ \dprBB{\frac{2\tau}{\sqrt\theta},h'} - g_c(h') \biggr\} = g_c^*\left(\frac{2\tau}{\sqrt\theta}\right) = \frac{4}{\theta}g_c^*(\tau), \] which is finite by the identity just proved and the estimate on $\rho^\circ$ in~\eqref{eq:jbar_est}. The same estimate implies that $g_c^*(\tau)\geq0$, with equality if and only if $\tau=0$. \end{proof}

\begin{defn} Let $\theta\in[0,1]$, $l>0$, $\tau\in\Msn$, and let $\mathbb X$ be a Hooke's law with $\mathbb X<\mathbb A_0$. We define 
	\[ K^\mathbb X_\theta := \setb{ \mathbb L_{\mathbb X,\mathbb A_0,1-\theta,v,R,\chi} }{ v,R,\chi \text{ vary as in Definition~\ref{defn:vasikotatos orismos}} }, \] 
\[ F^\mathbb X(\tau,\theta) := \inf_{\mathbb L\in K^\mathbb X_\theta} \dprb{\mathbb L \tau,\tau}, \]
and
\begin{equation}\label{FalFl} F_l^\mathbb X(\tau) := \inf_{\theta\in[0,1]} \bigl\{ F^\mathbb X(\tau,\theta)+l\theta \bigr\}, \qquad F_l(\tau) = \inf_{\theta\in[0,1]} \bigl\{ F(\tau,\theta)+l\theta \bigr\}. 
\end{equation}
\end{defn} 

{
\begin{lem}\label{realization}
	Let $\mathbb A$ and $\mathbb B$ be two Hooke's laws with $\mathbb A< \mathbb B$. Let $\mathbb L_p$ be a finite complementary \(p\)-laminate with
	phases $\mathbb A$ and $\mathbb B$ in the sense of Definition \ref{defn: complementary p-laminate}. More precisely, assume that
\(\mathbb L_p\) is obtained through the recursive construction
\[
\mathbb L_0:=\mathbb A^{-1},
\qquad
\mathbb L_j
:=
\mathbb L_{\mathbb L_{j-1}^{-1},\mathbb B,\theta_j,v_j,R_j,\chi_j},
\qquad j=1,\dots,p,
\]
with \(\theta_j\in(0,1)\). We denote by
\[
m_p:=\prod_{j=1}^p \theta_j
\]
the volume fraction of the original phase \(\mathbb A\) associated with this
recursive representation.
	Then there exists a sequence of periodic two-phase coefficients
\[
\widehat{\mathbb H}_k(y)=\chi_k(y)\mathbb A+(1-\chi_k(y))\mathbb B,
\qquad
\chi_k\colon\mathbb R^n\to\{0,1\},
\]
such that, on a fundamental periodicity cell \(Q_k\) with \(|Q_k|=1\), it holds
\[
\int_{Q_k}\chi_k(y)\,dy=m_p,
\]
and such that the corresponding complementary effective Hooke's laws
\(\mathbb L_k^{\mathrm{eff}}\), defined as in~\eqref{L}, satisfy
\[
\mathbb L_k^{\mathrm{eff}}\to \mathbb L_p
\]
in the finite-dimensional space of fourth-order tensors.
\end{lem}

\begin{proof}
The proof is by induction on \(p\).
	For \(p=1\), the statement is just the definition of a complementary
	$1$-laminate.  
    
    Assume that the claim is known for
	\((p-1)\)-laminates. By the recursive construction of finite complementary
	laminates \eqref{recursive definition}, the \(p\)-laminate \(\mathbb L_p\) is obtained by laminating the Hooke
	law \(\mathbb L_{p-1}^{-1}\) with the phase
	\(\mathbb B\). Namely,
\[
\mathbb L_p
=
\mathbb L_{\mathbb L_{p-1}^{-1},\mathbb B,\theta_p,v_p,R_p,\chi_p},
\]
where
\[
\int_{R_pQ}\chi_p(y)\,dy=\theta_p .
\]

By the induction hypothesis, there exists a sequence of periodic two-phase coefficients
\[
\mathbb H_k^{\rm in}(y)
=
\chi_k^{\rm in}(y)\mathbb A
+
\bigl(1-\chi_k^{\rm in}(y)\bigr)\mathbb B
\]
whose volume fraction of the phase \(\mathbb A\) is
\[
m_{p-1}:=\prod_{j=1}^{p-1}\theta_j,
\]
and whose complementary effective Hooke's laws
\(\mathbb L_k^{\rm in}\) satisfy
\[
\mathbb L_k^{\rm in}\to \mathbb L_{p-1}.
\]
In particular,
\[
(\mathbb L_k^{\rm in})^{-1}\to \mathbb L_{p-1}^{-1}.
\]
	The periodicity cell of \(\mathbb H_k^{\rm in}\) may be a rotated cell
	\(R_kQ\). 
	For fixed \(k\), introduce the two-scale coefficient
\[
\mathbb H_k(y,z)
=
\chi_p(y)\mathbb H_k^{\rm in}(z)
+
\bigl(1-\chi_p(y)\bigr)\mathbb B ,
\]
where \(y\in R_pQ\) is the slow variable and \(z\in R_kQ\) is the fast
variable. Equivalently,
\[
\mathbb H_k(y,z)
=
\chi_p(y)\chi_k^{\rm in}(z)\mathbb A
+
\bigl(1-\chi_p(y)\chi_k^{\rm in}(z)\bigr)\mathbb B .
\]
Hence \(\mathbb H_k\) takes only the two values \(\mathbb A\) and \(\mathbb B\).

For \(\delta>0\), set, inside the outer cell \(R_pQ\),
\[
\mathbb H_{k,\delta}^{\rm raw}(y)
:=
\mathbb H_k(y,y/\delta).
\]
This coefficient need not be \(R_pQ\)-periodic in the whole space, because
the fast lattice \(R_k\mathbb Z^n\) need not be commensurable with the
outer lattice \(R_p\mathbb Z^n\). We therefore define the admissible
coefficient directly on the cell \(R_pQ\) and then extend it periodically.

Let
\[
\eta_{k,\delta}(y)
:=
\chi_p(y)\chi_k^{\rm in}(y/\delta)
\]
be the characteristic function of the phase \(\mathbb A\) in
\(\mathbb H_{k,\delta}^{\rm raw}\). Since \(\chi_k^{\rm in}\) is periodic with
average \(m_{p-1}\), the Riemann-Lebesgue lemma
gives as $\delta\to0$
\[
\int_{R_pQ}\eta_{k,\delta}(y)\,dy
\longrightarrow
\left(\int_{R_pQ}\chi_p(y)\,dy\right)m_{p-1}
=
\theta_p m_{p-1}
=
m_p.
\]
We choose a characteristic function
\(\widetilde\eta_{k,\delta}:R_pQ\to\{0,1\}\) such that
\[
\int_{R_pQ}\widetilde\eta_{k,\delta}(y)\,dy=m_p
\]
and
\[
\bigl|\{\widetilde\eta_{k,\delta}\neq \eta_{k,\delta}\}\bigr|
\longrightarrow0
\qquad\text{as }\delta\downarrow0.
\]
This is done by changing \(\eta_{k,\delta}\) on a set of measure equal to
\[
\left|
\int_{R_pQ}\eta_{k,\delta}(y)\,dy-m_p
\right|,
\]
which tends to zero.

We then define, on \(R_pQ\),
\[
\mathbb H_{k,\delta}(y)
:=
\widetilde\eta_{k,\delta}(y)\mathbb A+
\bigl(1-\widetilde\eta_{k,\delta}(y)\bigr)\mathbb B,
\]
and extend \(\mathbb H_{k,\delta}\) \(R_pQ\)-periodically to the whole space. Thus
\(\mathbb H_{k,\delta}\) is an admissible periodic two-phase coefficient with
exact volume fraction \(m_p\) of the phase \(\mathbb A\). Let \(\mathbb L_{k,\delta}\) be its
complementary effective Hooke law, in the sense of \eqref{L}.

We now identify the limit of \(\mathbb L_{k,\delta}\) as \(\delta\downarrow0\), for
\(k\) fixed. The coefficient \(\mathbb H_k(y,z)\) falls within the class
covered by \cite[Remark 2.13(iii)]{AllaireBriane96}: its entries are finite
sums of products of \(L^\infty\)-periodic functions depending on the
separate variables \(y\) and \(z\). Adapting the two-scale homogenization result for locally periodic
\(L^\infty\)-coefficients \cite[Theorem 2.11 and Corollary 2.12]{AllaireBriane96} to the
symmetric-gradient case, we obtain that \(
 \mathbb H_{k,\delta}^{\rm raw}
 \)
$H$-converges in the sense of linear elasticity, as $\delta\to0$, to \(\overline{\mathbb H_k}\),
defined as follows: for a.e. \(y\in R_pQ\) and
every \(E\in\mathbb M^{n\times n}_{\rm sym}\),
\[
\bigl\langle \overline{\mathbb H_k}(y)E,E\bigr\rangle
:=
\inf_{w\in H^1_{\rm per}(R_kQ;\mathbb R^n)}
\int_{R_kQ}
\left\langle
\mathbb H_k(y,z)
\bigl(E+e(w)(z)\bigr),
E+e(w)(z)
\right\rangle\,\text{d} z .
\]

Since \(\chi_p\) takes only the values \(0\) and \(1\), the inner
cell problem gives
\[
\overline {\mathbb H_k}(y)
=
\chi_p(y)(\mathbb L_k^{\rm in})^{-1}
+
\bigl(1-\chi_p(y)\bigr)\mathbb B .
\]
Indeed, on the set \(\{\chi_p=1\}\), one has
\(\mathbb H_k(y,z)=\mathbb H_k^{\rm in}(z)\), whose associated primal effective Hooke law is, by duality,
\((\mathbb L_k^{\rm in})^{-1}\); on the set \(\{\chi_p=0\}\), the coefficient
is the constant Hooke law \(\mathbb B\).
We now have to pass from $\mathbb H_{k,\delta}^{\rm raw}$ to the $R_pQ$-periodic correction $\mathbb H_{k,\delta}$. 
By compactness of \(H\)-convergence in linearized elasticity, every sequence
\(\delta_j\downarrow0\) admits a subsequence such that
\(\mathbb H_{k,\delta_j}\) \(H\)-converges to some tensor field \(\mathbb G_k\).
Since
\[
|\mathbb H_{k,\delta_j}-\mathbb H_{k,\delta_j}^{\rm raw}|
\leq
C\mathbf 1_{\{\widetilde\eta_{k,\delta_j}\neq\eta_{k,\delta_j}\}},
\]
and the right-hand side converges strongly to \(0\) in \(L^1(R_pQ)\),
the stability result for \(H\)-convergence in linearized elasticity \cite{DonatoHaddadou07}
implies that \(\mathbb G_k=\overline{\mathbb H_k}\). Hence the whole family
\(\mathbb H_{k,\delta}\) \(H\)-converges to \(\overline{\mathbb H_k}\)
as $\delta\to0$.

For symmetric uniformly elliptic coefficients, \(H\)-convergence in
linearized elasticity yields the \(\Gamma\)-convergence of the associated
quadratic elastic energies. Applied to the periodic cell problems with
prescribed average strain \(E\), this implies convergence of the minimum
values.

For \(E\in\mathbb M^{n\times n}_{\rm sym}\), we denote the primal effective Hooke law
associated with \(\mathbb H_{k,\delta}\) in $R_pQ$ by
\[
\bigl\langle
\mathcal C(\mathbb H_{k,\delta})E,E
\bigr\rangle
:=
\inf_{\varphi\in H^1_{\rm per}(R_pQ;\mathbb R^n)}
\int_{R_pQ}
\left\langle
\mathbb H_{k,\delta}(y)
\bigl(E+e(\varphi)(y)\bigr),
E+e(\varphi)(y)
\right\rangle\,dy,
\]
and analogously for \(\overline{\mathbb H_k}\).

By equicoercivity and $\Gamma$-convergence, we have, for every
\(E\in\mathbb M^{n\times n}_{\rm sym}\),
\[
\bigl\langle \mathcal C(\mathbb H_{k,\delta})E,E\bigr\rangle
\to
\bigl\langle \mathcal C(\overline {\mathbb H_k})E,E\bigr\rangle
\]
as $\delta\to0$, and then
\[
\mathcal C(\mathbb H_{k,\delta})
\to
\mathcal C(\overline {\mathbb H_k})
\]
in the finite-dimensional space of fourth-order tensors.

Passing to the complementary formulation and using the standard duality
identity between the primal and complementary effective Hooke's laws, we get
\[
\mathbb L_{k,\delta}
=
\mathcal C(\mathbb H_{k,\delta})^{-1}
\to
\mathcal C(\overline{\mathbb H_k})^{-1}
=
\mathbb L_{(\mathbb L_k^{\rm in})^{-1},\mathbb B,\theta_p,v_p,R_p,\chi_p}
\]
as \(\delta\to0\).

Since \(\mathbb L_k^{\rm in}\to \mathbb L_{p-1}\) as \(k\to\infty\), and
the tensors are uniformly elliptic, we have
\[
(\mathbb L_k^{\rm in})^{-1}\to \mathbb L_{p-1}^{-1}.
\]
By the continuity of the complementary cell formula with respect to uniformly
elliptic phase tensors, it follows that
\[
\mathbb L_{(\mathbb L_k^{\rm in})^{-1},\mathbb B,\theta_p,v_p,R_p,\chi_p}
\to
\mathbb L_{\mathbb L_{p-1}^{-1},\mathbb B,\theta_p,v_p,R_p,\chi_p}
=
\mathbb L_p,
\]
where the last equality follows from the recursive definition of
\(\mathbb L_p\).

We finally choose the two parameters \(k\) and \(\delta\) along a diagonal
sequence: for every \(k\), we choose \(\delta_k>0\) so small that
\[
\mathcal C(\mathbb H_{k,\delta_k})^{-1}
-
\mathbb L_{(\mathbb L_k^{\rm in})^{-1},\mathbb B,\theta_p,v_p,R_p,\chi_p}
\to0 .
\]
Since
\[
\mathbb L_{(\mathbb L_k^{\rm in})^{-1},\mathbb B,\theta_p,v_p,R_p,\chi_p}
\to
\mathbb L_p,
\]
it follows that
\[
\mathbb L_k^{\rm eff}:=
\mathcal C(\mathbb H_{k,\delta_k})^{-1}
\to
\mathbb L_p .
\]
Thus the periodic two-phase coefficients \(\mathbb H_{k,\delta_k}\), which have
exact volume fraction \(m_p\) of the phase \(\mathbb A\), realize \(\mathbb L_p\). This
concludes the proof.
\end{proof}

\begin{lem} \label{lem:black box} For every $\tau\in\Msn$ and every $\theta\in[0,1]$, it holds \[ F^\mathbb X(\tau,\theta) = \dprb{\mathbb A_0^{-1}\tau,\tau} + (1-\theta) \sup_{h\in\Msn} \bigl\{ 2\dprb{\tau,h} - \dprb{(\mathbb X^{-1}-\mathbb A_0^{-1})^{-1}h,h} - \theta g_c(h) \bigr\}. \] \end{lem} \begin{proof} For $\theta=0$ and $\theta=1$ the formula follows directly from \[ K^\mathbb X_0=\{\mathbb X^{-1}\}, \qquad K^\mathbb X_1=\{\mathbb A_0^{-1}\}. \] We therefore assume $\theta\in(0,1)$. By Proposition~\ref{prop:sxesi 6.9}, applied with $\mathbb A=\mathbb X$, $\mathbb B=\mathbb A_0$, and volume fraction $1-\theta$, we have \[ F^\mathbb X(\tau,\theta) \geq \dprb{\mathbb A_0^{-1}\tau,\tau} + (1-\theta) \sup_{h\in\Msn} \bigl\{ 2\dprb{\tau,h} - \dprb{(\mathbb X^{-1}-\mathbb A_0^{-1})^{-1}h,h} - \theta g_c(h) \bigr\}. \] 

We next prove the opposite inequality. 
	By Proposition ~\ref{prop:proposition regarding optimality for H-S}, applied with volume fraction \(1-\theta\), there exists
	a complementary \(p\)-laminate \(\mathbb L_p\) with respect to \(\mathbb X\) and \(\mathbb A_0\)
	such that
	\begin{equation}\label{first inequality}
	\begin{aligned}
	\langle \mathbb L_p\tau,\tau\rangle
	&=
	\langle \mathbb A_0^{-1}\tau,\tau\rangle  \\
	&\quad
	+(1-\theta)
	\sup_{h\in\mathbb M^{n\times n}_{\rm sym}}
	\left\{
	2\langle\tau,h\rangle
	-
	\left\langle(\mathbb X^{-1}-\mathbb A_0^{-1})^{-1}h,h\right\rangle
	-
	\theta g_c(h)
	\right\}.
	\end{aligned}
	\end{equation}
	By Lemma \ref{realization} with
	$\mathbb A=\mathbb X$, $\mathbb B=\mathbb A_0$, $m_p=1-\theta$,
	 there exists a sequence
	\(\mathbb L_k\in K^\mathbb X_\theta\) such that
	\(
	\mathbb L_k\to \mathbb L_p .
	\)
	Hence, by the definition of \(F^\mathbb X(\tau,\theta)\) and by the continuity
	of the map \(\mathbb L\mapsto\langle \mathbb L\tau,\tau\rangle\),
	\[
	F^\mathbb X(\tau,\theta)
	=
	\inf_{\mathbb L\in K^\mathbb X_\theta}\langle \mathbb L\tau,\tau\rangle
	\leq
	\lim_{k\to\infty}\langle \mathbb L_k\tau,\tau\rangle
	=
	\langle \mathbb L_p\tau,\tau\rangle .
	\]
	Combining this with \eqref{first inequality} gives the claimed identity.
\end{proof}
\begin{prop} \label{prop:isotita sto orio kathos paei sto 0} The limit of $F^\mathbb X(\tau,\theta)$ as $\mathbb X\to0$ exists. Denoting it by $F(\tau,\theta)$, we have, for every $\theta>0$, \[ F(\tau,\theta) = \dprb{\mathbb A_0^{-1}\tau,\tau} + (1-\theta) \sup_{h\in\Msn} \bigl\{ 2\dprb{\tau,h} - \theta g_c(h) \bigr\}. \] Moreover, $F(\tau,\theta)$ is finite for $\theta>0$, while \[ F(0,\theta)=0 \qquad\text{for every }\theta\in[0,1], \] and \[ F(\tau,0)=+\infty \qquad\text{for every }\tau\neq0. \] \end{prop} \begin{proof} We first assume that $\theta>0$. By Lemma~\ref{lem:black box}, for every Hooke's law $\mathbb X<\mathbb A_0$, \[ F^\mathbb X(\tau,\theta) = \dprb{\mathbb A_0^{-1}\tau,\tau} + (1-\theta) \sup_{h\in\Msn} \bigl\{ 2\dprb{\tau,h} - \dprb{(\mathbb X^{-1}-\mathbb A_0^{-1})^{-1}h,h} - \theta g_c(h) \bigr\}. \] Set \[ B_\mathbb X:=(\mathbb X^{-1}-\mathbb A_0^{-1})^{-1}. \] Since $\mathbb X<\mathbb A_0$, the operator $B_\mathbb X$ is positive definite. Moreover, as $\mathbb X\to0$, \[ B_\mathbb X\to0. \] Indeed, for $\mathbb X$ small enough, \[ B_\mathbb X = \mathbb X(\mathbb I-\mathbb A_0^{-1}\mathbb X)^{-1}, \] and the inverse is uniformly bounded. Since $B_\mathbb X$ is positive definite, we immediately get \[ F^\mathbb X(\tau,\theta) \leq \dprb{\mathbb A_0^{-1}\tau,\tau} + (1-\theta) \sup_{h\in\Msn} \bigl\{ 2\dprb{\tau,h} - \theta g_c(h) \bigr\}. \] Taking the limsup yields \[ \limsup_{\mathbb X\to0}F^\mathbb X(\tau,\theta) \leq \dprb{\mathbb A_0^{-1}\tau,\tau} + (1-\theta) \sup_{h\in\Msn} \bigl\{ 2\dprb{\tau,h} - \theta g_c(h) \bigr\}. \]

Conversely, fixing $h\in\Msn$ and using again Lemma~\ref{lem:black box}, we have \[ F^\mathbb X(\tau,\theta) \geq \dprb{\mathbb A_0^{-1}\tau,\tau} + (1-\theta) \bigl\{ 2\dprb{\tau,h} - \dprb{B_\mathbb X h,h} - \theta g_c(h) \bigr\}. \] Letting $\mathbb X\to0$ and using $B_\mathbb X\to0$, we obtain \[ \liminf_{\mathbb X\to0}F^\mathbb X(\tau,\theta) \geq \dprb{\mathbb A_0^{-1}\tau,\tau} + (1-\theta) \bigl\{ 2\dprb{\tau,h} - \theta g_c(h) \bigr\}. \] Taking the supremum over $h\in\Msn$ gives \[ \liminf_{\mathbb X\to0}F^\mathbb X(\tau,\theta) \geq \dprb{\mathbb A_0^{-1}\tau,\tau} + (1-\theta) \sup_{h\in\Msn} \bigl\{ 2\dprb{\tau,h} - \theta g_c(h) \bigr\}. \] The desired formula follows. The finiteness for $\theta>0$ is a consequence of Proposition~\ref{prop:kyrio neo apotelesma tou paper}. It remains to consider $\theta=0$. By the definition of $K^\mathbb X_0$, we have \[ K^\mathbb X_0=\{\mathbb X^{-1}\}, \] and hence \[ F^\mathbb X(\tau,0)=\dprb{\mathbb X^{-1}\tau,\tau}. \] Therefore $F^\mathbb X(0,0)=0$ for every $\mathbb X$, while, if $\tau\neq0$, \[ \dprb{\mathbb X^{-1}\tau,\tau}\to+\infty \qquad\text{as }\mathbb X\to0. \] This proves the statement for $\theta=0$ and concludes the proof.\end{proof}

 \begin{prop} \label{prop:sxetika me tin syglisi sto Fl+} For every $\tau\in\Msn$, $F_l^\mathbb X(\tau)$ converges to $F_l(\tau)$ as $\mathbb X\to0$, where $F^\mathbb X_l$ and $F_l$ have been defined in~\eqref{FalFl}. Moreover, $F_l(\tau)$ is finite. \end{prop} \begin{proof} Since $F_l(\tau)\leq F(\tau,1)+l<+\infty$ and $F_l(\tau)\geq0$, the quantity $F_l(\tau)$ is finite. Fix $\theta\in[0,1]$. By the definition of $F_l^\mathbb X$, \[ F_l^\mathbb X(\tau)\leq F^\mathbb X(\tau,\theta)+l\theta. \] Taking the limsup as $\mathbb X\to0$ and using Proposition~\ref{prop:isotita sto orio kathos paei sto 0}, we obtain \[ \limsup_{\mathbb X\to0}F_l^\mathbb X(\tau) \leq F(\tau,\theta)+l\theta. \] Taking the infimum over $\theta\in[0,1]$ gives \[ \limsup_{\mathbb X\to0}F_l^\mathbb X(\tau)\leq F_l(\tau). \]

We now prove the {other inequality}. Let $\mathbb X_n\to0$ be such that \[ \liminf_{\mathbb X\to0}F_l^\mathbb X(\tau) = \lim_{n\to\infty}F_l^{\mathbb X_n}(\tau). \] For every $n$, choose $\theta_n\in[0,1]$ such that \[ F_l^{\mathbb X_n}(\tau)+\frac1n \geq F^{\mathbb X_n}(\tau,\theta_n)+l\theta_n. \] Up to a subsequence, we may assume that $\theta_n\to\theta_0\in[0,1]$. Set \[ B_{\mathbb X_n}:= (\mathbb X_n^{-1}-\mathbb A_0^{-1})^{-1}. \] As in the proof of Proposition~\ref{prop:isotita sto orio kathos paei sto 0}, we have \[ B_{\mathbb X_n}\to0. \] By Proposition~\ref{prop:sxesi 6.9}, for every $h\in\Msn$, \[ F^{\mathbb X_n}(\tau,\theta_n) \geq \dprb{\mathbb A_0^{-1}\tau,\tau} + (1-\theta_n) \bigl\{ 2\dprb{\tau,h} - \dprb{B_{\mathbb X_n}h,h} - \theta_n g_c(h) \bigr\}. \] Adding $l\theta_n$ and passing to the limit gives \[ \begin{aligned} \liminf_{\mathbb X\to0}F_l^\mathbb X(\tau) &\geq \dprb{\mathbb A_0^{-1}\tau,\tau} + (1-\theta_0) \bigl\{ 2\dprb{\tau,h} - \theta_0 g_c(h) \bigr\} + l\theta_0. \end{aligned} \] Since this holds for every $h\in\Msn$, we obtain \[ \liminf_{\mathbb X\to0}F_l^\mathbb X(\tau) \geq F(\tau,\theta_0)+l\theta_0. \] Therefore, \[ \liminf_{\mathbb X\to0}F_l^\mathbb X(\tau) \geq F_l(\tau). \] Combining the limsup and liminf inequalities yields \[ \lim_{\mathbb X\to0}F_l^\mathbb X(\tau)=F_l(\tau), \] as claimed. \end{proof}

\begin{prop} \label{peri Fl} For every $\tau\in\Msn$, {it holds that} \[ F_l(\tau)= \begin{cases} \dprb{\mathbb A_0^{-1}\tau,\tau}+l & \text{if } \rho_l(\tau)\geq1,\\[3pt] \dprb{\mathbb A_0^{-1}\tau,\tau} + l\rho_l(\tau)\bigl(2-\rho_l(\tau)\bigr) & \text{if } \rho_l(\tau)\leq1, \end{cases} \] where \[ \rho_l(\tau):=\frac{2}{\sqrt l}\sqrt{g_c^*(\tau)}, \qquad g_c^*(\tau):= \sup_{h\in\Msn} \bigl\{ \dprb{\tau,h}-g_c(h) \bigr\}. \] \end{prop} \begin{proof} The formula is stated without proof in~\cite[(7.4)]{AllaireKohn93}. We include the proof for completeness. The case $\tau=0$ is immediate. Indeed, by Proposition~\ref{prop:isotita sto orio kathos paei sto 0}, \[ F_l(0) = \inf_{\theta\in[0,1]} \bigl\{ F(0,\theta)+l\theta \bigr\} = \inf_{\theta\in[0,1]}l\theta = 0, \] and $\rho_l(0)=0$.

We now assume that $\tau\neq0$. Since $F(\tau,0)=+\infty$, again by Proposition~\ref{prop:isotita sto orio kathos paei sto 0} we have \[ F_l(\tau) = \inf_{\theta\in(0,1]} \bigl\{ F(\tau,\theta)+l\theta \bigr\}. \] For $\theta\in(0,1]$, Proposition~\ref{prop:isotita sto orio kathos paei sto 0} gives \[ F(\tau,\theta) = \dprb{\mathbb A_0^{-1}\tau,\tau} + (1-\theta) \sup_{h\in\Msn} \bigl\{ 2\dprb{\tau,h} - \theta g_c(h) \bigr\}. \] Since $g_c$ is positively $2$-homogeneous, \[ \sup_{h\in\Msn} \bigl\{ 2\dprb{\tau,h} - \theta g_c(h) \bigr\} = \frac4\theta g_c^*(\tau). \] Therefore, \[ F_l(\tau) = \inf_{\theta\in(0,1]} \left\{ \dprb{\mathbb A_0^{-1}\tau,\tau} + \frac{4(1-\theta)}{\theta}g_c^*(\tau) + l\theta \right\}. \] Set \[ a:=\dprb{\mathbb A_0^{-1}\tau,\tau}, \qquad b:=4g_c^*(\tau). \] Then, \[ F_l(\tau) = \inf_{\theta\in(0,1]} \left\{ a+\frac{b(1-\theta)}{\theta}+l\theta \right\} = a-b+\inf_{\theta\in(0,1]} \left\{ \frac b\theta+l\theta \right\}. \] Since $\tau\neq0$, Proposition~\ref{prop:kyrio neo apotelesma tou paper} gives $g_c^*(\tau)>0$, hence $b>0$. The function \[ \theta\mapsto \frac b\theta+l\theta \] is minimized on $(0,1]$ at \[ \theta_*=\min\left\{1,\sqrt{\frac bl}\right\}. \] By the definition of $\rho_l$, we have \[ \sqrt{\frac bl} = \frac{2}{\sqrt l}\sqrt{g_c^*(\tau)} = \rho_l(\tau). \] If $\rho_l(\tau)\geq1$, then $\theta_*=1$, and hence \[ F_l(\tau)=a-b+b+l=a+l = \dprb{\mathbb A_0^{-1}\tau,\tau}+l. \] If $\rho_l(\tau)\leq1$, then $\theta_*=\rho_l(\tau)$, and using \[ b=l\rho_l(\tau)^2, \] we get \[ \begin{aligned} F_l(\tau) &= a+\frac{b(1-\rho_l(\tau))}{\rho_l(\tau)} + l\rho_l(\tau) \\ &= a+l\rho_l(\tau)(1-\rho_l(\tau)) + l\rho_l(\tau) \\ &= a+l\rho_l(\tau)\bigl(2-\rho_l(\tau)\bigr). \end{aligned} \] This proves the formula. \end{proof}

The following results concern the well-known notion of (symmetric) $\dive$-quasiconvexity~\cite{FonsecaMuller99}, but we give a complete account because it is important to define the notions in the correct order, which is not always clear in the literature (also see the remarks after the statement of Proposition~\ref{peri Fl kai Qdiv}).

\begin{defn}
A locally bounded Borel measurable function $f \colon \Msn \to \R$ is said {to be} \emph{symmetric $\dive$-quasiconvex} if for any $\varphi  \in\Crm^\infty(\R^n;\Msn)$ with $\varphi(y+e_{j})=\varphi(y)$ for all $j \in \left\{1,2,...,n\right\}$ and every $y \in \R^n$, such that $\dive \varphi=0$ in $\R^n$, it holds {that}
$$f\biggl(\int_{Q} \varphi(x) \dd x\biggr) \leq \int_{Q}f(\varphi(x)) \dd x.$$
\end{defn}
By the standard cell-independence property of $\mathcal A$-quasiconvexity for constant-rank operators, see~\cite[Lemma 8]{Raita2019}, the cube $Q$ in the definition may be replaced by any rotated cube $RQ$, with periodicity understood with respect to the lattice generated by $Re_1,\dots,Re_n$.} 

\begin{defn}
 Let $h \colon \Msn\to \R$ be a locally bounded, bounded below, Borel measurable function. We define, the symmetric div-quasiconvex envelope of $h$ as
 $$Qh(\tau):=\sup \, \setB{g(\tau)}{ g \text{ is  \emph{(symmetric) $\dive$-quasiconvex} and } g\leq h},$$ for any $\tau \in \Msn.$
\end{defn}

\begin{lem}
	\label{lem:sxetika me rank n-1 klp}
	Let $f\colon \Msn\to\R$ be a locally bounded Borel measurable function {that} is symmetric $\dive$-quasiconvex. Then $f$ is continuous. Moreover, let $h\colon\Msn\to\R$ be locally bounded, bounded below, and Borel measurable. Then $Qh$ is symmetric $\dive$-quasiconvex.
\end{lem}

\begin{proof}
By~\cite[Lemma 2.4]{ContiMullerOrtiz20}, every locally bounded symmetric $\dive$-quasiconvex function is locally Lipschitz continuous. This proves the first assertion.

We turn to the second assertion. Since $h$ is bounded below, there exists $c\in\R$ such that $h\geq c$. Hence
\[
c\leq Qh\leq h.
\]
In particular, $Qh$ is bounded below and locally bounded. By the first part, every symmetric $\dive$-quasiconvex minorant of $h$ is continuous. Therefore $Qh$, being the supremum of such minorants, is lower semicontinuous, hence Borel measurable.

Finally, the defining inequality for symmetric $\dive$-quasiconvexity is stable under taking suprema. Hence $Qh$ is symmetric $\dive$-quasiconvex.
\end{proof}

\begin{prop} \label{prop:anaparastasi tis thikis} Let $h\colon\Msn\to\R$ be continuous {and bounded from below}. Then, for every $\tau\in\Msn$, the following representation formula holds: 
\begin{align}Qh(\tau)=\inf \, \setBB{\int_{Q} h(\varphi(y))\dd y }{ & \varphi\in\Crm^\infty_{\rm per}(Q;\Msn), \notag\\
	&\int_{Q}\varphi \dd y=\tau,\; \dive\varphi=0 \text{ in $\R^n$ } }. \end{align}
 \end{prop} 
{As before,} the set $Q$ in the representation formula above can be replaced by any rotated cube $RQ$, and in this case the periodicity of $\varphi$ is understood with respect to the basis $Re_1,\dots,Re_n$.
\begin{proof}
Denote by {$R(\tau)$} the right-hand side of the formula. By~\cite[Proposition~3.4]{FonsecaMuller99}, applied to the symmetric divergence constraint, the function $R$ is symmetric $\dive$-quasiconvex. We first prove that $Qh\leq R$. Let $g$ be any symmetric $\dive$-quasiconvex function such that $g\leq h$. If  $\varphi\in\Crm^\infty_{\rm per}(Q;\Msn)$ satisfies \[ \dive\varphi=0\quad\text{in }\R^n, \qquad \int_{Q}\varphi\,\dd x=\tau, \] then, by symmetric $\dive$-quasiconvexity of $g$, \[ g(\tau) = g\left(\int_{Q}\varphi\,\dd x\right) \leq \int_{Q} g(\varphi)\,\dd x \leq \int_{Q} h(\varphi)\,\dd x. \] Taking the infimum over all admissible field $\varphi$ gives \( g(\tau)\leq R(\tau). \) Taking the supremum over all such minorants $g$ yields \(Qh(\tau)\leq R(\tau). \) Conversely, by choosing the constant admissible field $\varphi\equiv\tau$, we get \( R(\tau)\leq h(\tau). \) Since $R$ is symmetric $\dive$-quasiconvex and lies below $h$, the definition of $Qh$ gives \( R(\tau)\leq Qh(\tau). \) Combining the two inequalities, we conclude that $Qh(\tau)=R(\tau)$. 
\end{proof}

For $\tau\in\Msn$, let us define
 \[ f_l^\mathbb X(\tau) := \min \left\{ \dprb{\mathbb A_0^{-1}\tau,\tau}+l, \dprb{\mathbb X^{-1}\tau,\tau} \right\}\]
 and  
\begin{equation} \label{eq:f_l} f_l(\tau) := \begin{cases} \dprb{\mathbb A_0^{-1}\tau,\tau}+l & \text{if } \tau\neq0,\\ 0 & \text{if } \tau=0. \end{cases} \end{equation}
The following lemma identifies the functions $F_l^\mathbb X$ defined in~\eqref{FalFl} with the symmetric $\dive$-quasiconvex envelopes of the $f_l^\mathbb X$.

\begin{lem} \label{lem:isotita sto gqiv me to a stathero} For every $\tau\in\Msn$ we have \[ F_l^\mathbb X(\tau)=Qf_l^\mathbb X(\tau). \] \end{lem} \begin{proof} Since $f_l^\mathbb X$ is continuous, in view of Proposition~\ref{prop:anaparastasi tis thikis} it is enough to prove that, for every $\tau\in\Msn$, \begin{align}F^{ \mathbb X}_{l}(\tau)=\inf \, \setBB{\int_{Q} f^{ \mathbb X}_{l}(\varphi(y))\dd y }{ &\varphi\in C^\infty_{\per}(Q;\Msn), \notag\\
		&\int_{Q}\varphi \dd y=\tau,\; \dive\varphi=0 \text{ in $\R^n$ } }, \end{align}
	We denote the right-hand side by {$R(\tau)$.} We first prove that \( F_l^\mathbb X(\tau)\leq R(\tau). \) Let $\varphi\in\Crm^\infty_{\per}(Q;\Msn)$ be admissible in the definition of $R(\tau)$. Define $\chi\colon\R^n\to\{0,1\}$ by \[ \chi(y):= \begin{cases} 1 & \text{if } \dprb{\mathbb X^{-1}\varphi(y),\varphi(y)} \leq \dprb{\mathbb A_0^{-1}\varphi(y),\varphi(y)}+l,\\ 0 & \text{otherwise.} \end{cases} \] Then $\chi$ is Borel measurable and $Q$-periodic. Set \[ 1-\theta:=\int_Q\chi(y)\,\dd y, \] so that $\theta\in[0,1]$. Let \(\mathbb H(y):=\chi(y)\mathbb X+(1-\chi(y))\mathbb A_0. \) By the definition of $\chi$, \[ f_l^\mathbb X(\varphi(y)) = \dprb{\mathbb H(y)^{-1}\varphi(y),\varphi(y)} + l(1-\chi(y)). \] Therefore \begin{equation}\label{fal} \int_Q f_l^\mathbb X(\varphi(y))\,\dd y = \int_Q \dprb{\mathbb H(y)^{-1}\varphi(y),\varphi(y)} \,\dd y + l\theta. \end{equation} Let \(\mathbb L:=\mathbb L_{\mathbb X,\mathbb A_0,1-\theta,e_1,I,\chi}. \) Then $\mathbb L\in K_\theta^\mathbb X$. Moreover, $\varphi$ is admissible in the cell formula defining $\mathbb L$ {(see~\eqref{L})}, and hence, by the definition of $\mathbb L$, \[ \dprb{\mathbb L \tau,\tau} \leq \int_Q \dprb{\mathbb H(y)^{-1}\varphi(y),\varphi(y)} \,\dd y. \] Using~\eqref{FalFl} and~\eqref{fal}, we obtain \[ F_l^\mathbb X(\tau)\leq F^\mathbb X(\tau,\theta)+l\theta \leq \dprb{\mathbb L \tau,\tau}+l\theta \leq \int_Q f_l^\mathbb X(\varphi(y))\,\dd y. \] Taking the infimum over all admissible $\varphi$ gives \( F_l^\mathbb X(\tau)\leq R(\tau). \) We now prove the opposite inequality. Fix $\theta\in[0,1]$. We prove that \( R(\tau)\leq F^\mathbb X(\tau,\theta)+l\theta. \) Let $\mathbb L\in K_\theta^\mathbb X$. Then \( \mathbb L=\mathbb L_{\mathbb X,\mathbb A_0,1-\theta,v,R,\chi} \) for some admissible $v,R,\chi$. In particular, \[ \int_{Q_{v,R}}\chi(y)\,\dd y=1-\theta. \] Set \(\mathbb H(y):=\chi(y)\mathbb X+(1-\chi(y))\mathbb A_0. \) Let $\tilde\varphi$ be admissible in the cell formula defining $\langle\mathbb L \tau,\tau\rangle$ (see~\eqref{L}). Then, { 
	\begin{align*} &\int_{Q_{v,R}} \dprb{\mathbb H(y)^{-1}\tilde\varphi(y),\tilde\varphi(y)} \,\dd y + l\theta \\ &\quad = \int_{Q_{v,R}} \chi(y)\dprb{\mathbb X^{-1}\tilde\varphi(y),\tilde\varphi(y)} \,\dd y  + \int_{Q_{v,R}} (1-\chi(y)) \left[ \dprb{\mathbb A_0^{-1}\tilde\varphi(y),\tilde\varphi(y)} + l \right] \,\dd y \\
	 &\quad \geq \int_{Q_{v,R}} f_l^\mathbb X(\tilde\varphi(y)) \,\dd y  \\
     &\quad \geq R(\tau). \end{align*}} In the last inequality we used the representation formula for $R$ on the rotated cell $Q_{v,R}$. Taking the infimum over all admissible $\tilde\varphi$ in the cell formula for $\mathbb L$, we get \( \dprb{\mathbb L \tau,\tau}+l\theta\geq R(\tau). \) Then taking the infimum over $\mathbb L\in K_\theta^\mathbb X$ yields \( F^\mathbb X(\tau,\theta)+l\theta\geq R(\tau). \) Finally, taking the infimum over $\theta\in[0,1]$, we obtain \(F_l^\mathbb X(\tau)\geq R(\tau). \)Combining the two inequalities gives {\( F_l^\mathbb X=R=Qf_l^\mathbb X, \)} as claimed. \end{proof}

\begin{prop} \label{peri Fl kai Qdiv}

For any $\tau \in \Msn$ it holds that $F_{l}(\tau)=Qf_{l}(\tau).$

\end{prop}

The proof is essentially the same as the first part of~\cite[Theorem 3.1]{AllaireBonnetierFrancfortJouve97} and will not be repeated here. Indeed, neither the dimension $n$ nor the fact that $j$ is a general quadratic form plays a role in this argument. We only point out the following technical issue. In~\cite[Theorem 3.1]{AllaireBonnetierFrancfortJouve97}, the continuity of $Qf_l$ is used in an essential way. There, this is justified by observing that $Qf_l$ is symmetric $\dive$-quasiconvex and hence rank-$(n-1)$-convex, which implies continuity. However, in order to prove that $Qf_l$ is symmetric $\dive$-quasiconvex, one first needs to know that $Qf_l$ is Borel measurable. This does not seem to follow immediately from the definition, and it is not discussed in~\cite{AllaireBonnetierFrancfortJouve97}. Lemma~\ref{lem:sxetika me rank n-1 klp} fills this gap and also provides a simpler proof of the continuity of $Qf_l$.

\subsection{{The relaxation of the generalized Kohn--Strang integrand}}

In~\cite{BabadjianIurlanoRindler23}, where the quadratic form is $j_0=\frac12|\cdot|^2$ and the dimension is restricted to $2$ or $3$, the following \emph{Kohn--Strang {integrand}} was introduced. Let $\alpha>0$, $\beta>0$, and define $h_0\colon\Msn\to\R$ by \[ h_0(\tau):= \begin{cases} \alpha|\tau|^2+\beta & \text{if } \tau\neq0,\\ 0 & \text{if } \tau=0. \end{cases} \] As recalled in that paper, explicit formulas for $Qh_0(\tau)$ are available in terms of the eigenvalues of $\tau\in\Msn$. These formulas were used to study the limiting behaviour of $Qh_{0,\e}$, where \[ h_{0,\e}(\tau):= \begin{cases} \displaystyle \frac{\e}{2}|\tau|^2+\frac{1}{2\e} & \text{if } \tau\neq0,\\ 0 & \text{if } \tau=0. \end{cases} \] More precisely, it is proved that \[ Qh_{0,\e}\to\rho_0^\circ \qquad\text{pointwise as } \e\todown0, \] where $\rho_0^\circ$ is the polar function associated with the quadratic form $j_0$. We follow the same strategy for a general quadratic form $j$ and in arbitrary dimension $n$. We therefore begin by introducing the corresponding Kohn--Strang {integrand} in this more general setting.

\begin{defn}
Let $\alpha>0$, $\beta>0$ and let $h \colon \Msn \to \R$ be the {\emph{generalized Kohn--Strang integrand},} defined by
$$h(\tau):=
\begin{cases}
\alpha j^*(\tau)+\beta & \text{if }\tau \neq 0,\\
0 & \text{if } \tau=0.
\end{cases}$$
\end{defn}
For any $\varepsilon>0$ we take $\alpha := \varepsilon$ and $\beta := \frac{1}{2\varepsilon}$  and define $$h_{\varepsilon}(\tau):=
\begin{cases}
\varepsilon j^*(\tau)+\frac{1}{2\varepsilon} & \text{if }\tau \neq 0,\\
0 & \text{if } \tau=0.
\end{cases}$$
for any $\tau \in \Msn.$ 

Using the continuity and symmetric $\dive$-quasiconvexity of $Qh$ (see Proposition ~\ref{prop:anaparastasi tis thikis}) and employing the same arguments as in the proof of ~\cite[Proposition 4.2]{BabadjianIurlanoRindler23} we have the following.
\begin{prop} \label{Proposition 4.2 from Filip's paper} For all $\sigma \in \Lrm^2(\Omega;\Msn)$, the following assertions hold:
\begin{enumerate}[(i)]
\item For all sequences $\{\sigma_k\}_{k\in \N} \subset \Lrm^2(\Omega;\Msn)$ such that $\sigma_k\wto\sigma$ in $\Lrm^2(\Omega;\Msn)$ and $\dive\sigma_k\to\dive\sigma$ in $[\Hrm^1(\Omega;\R^n)]^*$,
$$\liminf_{k\to\infty}\int_\Omega h(\sigma_k)\dd x\geq\int_\Omega Qh(\sigma)\dd x.$$
\item There exists a sequence $\{\bar \sigma_k\}_{k\in\N} \subset \Lrm^2(\Omega;\Msn)$ such that $\bar \sigma_k\wto\sigma$ in $\Lrm^2(\Omega;\Msn)$, $\dive\bar\sigma_k\to\dive\sigma$ in $[\Hrm^1(\Omega;\R^n)]^*$, and
$$\lim_{k\to\infty}\int_\Omega h(\bar\sigma_k)\dd x=\int_\Omega Qh(\sigma)\dd x.$$
\end{enumerate}
\end{prop}

\begin{prop} \label{praktika to neo apotelesma kapa} For every $\tau\in\Msn$, \[ Qh_{\e}(\tau)\to 2\sqrt2\,\sqrt{g_c^*(\tau)} \qquad\text{as }\e\todown0. \] 
\end{prop} 
\begin{proof} Since \( h_\e(\tau)=f_{\frac1{2\e}}(\sqrt\e\,\tau) \), {see~\eqref{eq:f_l},} we have \( Qh_\e(\tau)=Qf_{\frac1{2\e}}(\sqrt\e\,\tau). \) By Proposition~\ref{peri Fl kai Qdiv}, \( Qf_{\frac1{2\e}}(\sqrt\e\,\tau) = F_{\frac1{2\e}}(\sqrt\e\,\tau). \) We now apply Proposition~\ref{peri Fl} with $l=\frac1{2\e}$ and with $\tau$ replaced by $\sqrt\e\,\tau$. Since $g_c^*$ is positively $2$-homogeneous, we get \( g_c^*(\sqrt\e\,\tau)=\e g_c^*(\tau), \) and hence \( \rho_l(\sqrt\e\,\tau) = 2\sqrt2\,\e\,\sqrt{g_c^*(\tau)}, \) where $\rho_l$ has been defined in Proposition~\ref{peri Fl}.
	 Therefore, for every fixed $\tau$, the second branch in Proposition~\ref{peri Fl} applies for all sufficiently small $\e$, and we obtain \[ Qh_\e(\tau) = \e\dprb{\mathbb A_0^{-1}\tau,\tau} + 2\sqrt2\,\sqrt{g_c^*(\tau)} - 4\e g_c^*(\tau). \] Letting $\e\todown0$ gives \( Qh_\e(\tau)\to 2\sqrt2\,\sqrt{g_c^*(\tau)}. \)
\end{proof}

\begin{prop}
\label{prop:teliko apotelesma tou olou paper}
For any $\tau \in \Msn$, $Qh_{\varepsilon}(\tau)$ converges to  $\rho^\circ(\tau), \text{as  }\e \to 0,$ where $\rho^\circ$ has been defined in {Section~\ref{sc:convex}}.
\end{prop}
\begin{proof}
The proof follows by the above proposition and Proposition ~\ref{prop:kyrio neo apotelesma tou paper}.
\end{proof}

We are now in a position to prove the following proposition. \begin{prop}\label{prop:Olbermann}
Given a bounded $\Crm^2$-domain  $\Omega$ in $\R^n$ ($n\geq2$), let $f \in \M(\R^n;\R^n) \cap \Hrm^{-1}(\R^n;\R^n)$ be such that $\supp(f)\subset\overline\Omega$ and $f(r)=0$ for all $r \in \mathcal R$. For every $\lambda \in \M(\overline\Omega;\Msn)$ satisfying $-\dive\lambda=f$ in $\Dcal'(\R^n;\R^n)$, there exists a sequence $\{\lambda_\e\}_{\e>0}$ in $\Lrm^2(\Omega;\Msn)$ such that $\lambda_\e \wsto \lambda$ in $\M(\R^n;\Msn)$, $-\dive\lambda_\e \to f$ in $\Hrm^{-1}(\R^n;\R^n)$, and
$$\lim_{\e \todown 0} \int_\Omega h_\e(\lambda_\e)\dd x = \int_{\R^n} \rho^\circ \left(\frac{\di\lambda}{\di|\lambda|}\right)\dd|\lambda|.$$
\end{prop}
The strategy of the proof is parallel to that of ~\cite[Proposition 4.4]{BabadjianIurlanoRindler23} with small modifications and we only sketch it. More precisely, we use Propositions~\ref{prop:riesz type theorem}, ~\ref{Proposition 4.2 from Filip's paper}, and~\ref{prop:teliko apotelesma tou olou paper}, in place of  ~\cite[Propositions 4.2, 4.4]{BabadjianIurlanoRindler23} and the convergence of $Qh_{0,\e}$ discussed in the same paper.

\section{Proof of the main theorem} \label{sc:upper}
{We now finally come to the proof of Theorem~\ref{thm:conv-min}.}
\begin{proof}[Proof of Theorem~\ref{thm:conv-min}] The proof is an adaptation of~\cite[Theorem 1.1]{BabadjianIurlanoRindler23}. We may assume that $f\not\equiv0$. Indeed, if {$f \equiv 0$}, then $\mathscr C_\e(\e^{-1}\LL^n\res\omega)=0$ for every $\omega\in\A_\e$, while $\overline{\mathscr C}(\mu)=0$ for every $\mu\in\M^1(\overline\Omega)$, and the conclusion is immediate.

\emph{Step 1: Recovery sequence at a minimizer of the limit problem.} Consider \begin{equation}\label{eq:Olb} \kappa := \inf_\lambda \setBB{ \int_{\R^n} \rho^\circ\left(\frac{\di\lambda}{\di|\lambda|}\right)\dd|\lambda| }{ \lambda\in\M(\overline\Omega;\Msn),\ -\dive\lambda=f }, \end{equation} where the divergence constraint is understood in $\Dcal'(\R^n;\R^n)$. The admissible class is nonempty by Proposition~\ref{prop:riesz type theorem}. Moreover, by~\eqref{eq:jbar_est}, every admissible $\lambda$ satisfies {\[ \int_{\R^n} \rho^\circ\left(\frac{\di\lambda}{\di|\lambda|}\right)\dd|\lambda| \geq c|\lambda|(\R^n), \]} and the direct method, together with Reshetnyak's lower semicontinuity theorem, gives the existence of a minimizer for~\eqref{eq:Olb}. Since $f\not\equiv0$, we have $\kappa>0$. By~\cite[Theorem~2.3]{BouchitteButtazzo01}, applied to $\bar j$, there exist $\mu^*\in\M^1(\overline\Omega)$ and $\sigma^*\in\Lrm^1(\R^n,\mu^*;\Msn)$ such that \[ -\dive(\sigma^*\mu^*)=f, \] \[ \overline{\mathscr C}(\mu^*) = \min_{\mu\in\M^1(\overline\Omega)}\overline{\mathscr C}(\mu) = \frac{\kappa^2}{2}, \] and \begin{equation}\label{eq:kappa} \rho^\circ(\sigma^*)=\kappa \qquad \mu^*\text{-a.e.} \end{equation} In particular, $\sigma^*\in\Lrm^\infty(\R^n,\mu^*;\Msn)$ by~\eqref{eq:jbar_est}, $\sigma^*$ is a minimizer in the definition of $\overline{\mathscr C}(\mu^*)$, and \[ \lambda^*:=\sigma^*\mu^* \] is a minimizer of~\eqref{eq:Olb}. Fix $m\in(0,1)$. Applying Proposition~\ref{prop:Olbermann} with \( \e':=\frac{\e\sqrt m}{\kappa}, \) we obtain a family $\lambda_\e^m\in\Lrm^2(\Omega;\Msn)$ such that \[ \lambda_\e^m\wsto\lambda^* \qquad\text{in }\M(\R^n;\Msn), \] \[ \|-\dive\lambda_\e^m-f\|_{\Hrm^{-1}(\R^n;\R^n)}\to0, \] and \begin{equation}\label{eq:hme_conv} \lim_{\e\todown0} \int_\Omega h_\e^m(\lambda_\e^m)\dd x = \frac{\kappa}{\sqrt m} \int_{\overline\Omega} \rho^\circ\left(\frac{\di\lambda^*}{\di|\lambda^*|}\right)\dd|\lambda^*|, \end{equation} where \[ h_\e^m(\tau):= \begin{cases} \displaystyle \e j^*(\tau)+\frac{\kappa^2}{2m\e} & \text{if }\tau\neq0,\\[4pt] 0 & \text{if }\tau=0. \end{cases} \] Since $\lambda^*=\sigma^*\mu^*$, the positive one-homogeneity of $\rho^\circ$ and~\eqref{eq:kappa} give \[ \int_{\overline\Omega} \rho^\circ\left(\frac{\di\lambda^*}{\di|\lambda^*|}\right)\dd|\lambda^*| = \int_{\overline\Omega}\rho^\circ(\sigma^*)\,\dd\mu^* = \kappa. \] As in the proof of~\cite[Theorem 1.1]{BabadjianIurlanoRindler23}, the fields $\lambda_\e^m$ can be modified by {a} standard volume and equilibrium correction argument {(which is detailed in~\cite{BabadjianIurlanoRindler23})}. Letting first $\e\todown0$ and then $m\uparrow1$, {in conjunction with} a diagonal argument, we obtain $\omega_\e^*\in\A_\e$ and \[ \mu_\e^*:=\frac{\LL^n\res\omega_\e^*}{\e}\in\M^1(\overline\Omega), \qquad \mu_\e^*\wsto\mu^* \quad\text{in }\M(\R^n), \] together with admissible stresses $\sigma_\e^*\in\Lrm^2(\R^n,\mu_\e^*;\Msn)$ satisfying \[ -\dive(\sigma_\e^*\mu_\e^*)=f \] and \begin{equation}\label{eq:C3} \limsup_{\e\todown0} \int_\Omega j^*(\sigma_\e^*)\,\dd\mu_\e^* \leq \overline{\mathscr C}(\mu^*). \end{equation} Consequently, \begin{equation}\label{eq:recovery} \lim_{\e\todown0}\mathscr C_\e(\mu_\e^*) = \overline{\mathscr C}(\mu^*) = \min_{\mu\in\M^1(\overline\Omega)}\overline{\mathscr C}(\mu). \end{equation} Indeed, the upper bound follows from~\eqref{eq:C3}, while the opposite inequality follows from Proposition~\ref{prop:liminf}. 

\emph{Step 2: Almost minimizers and compactness.} Let $\alpha_\e\todown0$ and let $\omega_\e\in\A_\e$ be such that \[ \mathscr C_\e\left(\frac{\LL^n\res\omega_\e}{\e}\right) \leq \inf_{\omega\in\A_\e} \mathscr C_\e\left(\frac{\LL^n\res\omega}{\e}\right) + \alpha_\e. \] Set \[ \bar\mu_\e:=\frac{\LL^n\res\omega_\e}{\e}. \] Up to a subsequence, $\bar\mu_\e\wsto\bar\mu$ in $\M(\R^n)$ for some $\bar\mu\in\M^1(\overline\Omega)$. By Proposition~\ref{prop:liminf}, \[ \overline{\mathscr C}(\bar\mu) \leq \liminf_{\e\todown0}\mathscr C_\e(\bar\mu_\e). \] On the other hand, by almost minimality and by the recovery sequence constructed above, \[ \limsup_{\e\todown0}\mathscr C_\e(\bar\mu_\e) \leq \limsup_{\e\todown0}\inf_{\mu\in\M^1(\overline\Omega)}\mathscr C_\e(\mu) \leq \limsup_{\e\todown0}\mathscr C_\e(\mu_\e^*) = \overline{\mathscr C}(\mu^*). \] Since $\mu^*$ minimizes $\overline{\mathscr C}$, we get the chain of inequalities \[ \min_{\M^1(\overline\Omega)}\overline{\mathscr C} \leq \overline{\mathscr C}(\bar\mu) \leq \liminf_{\e\todown0}\mathscr C_\e(\bar\mu_\e) \leq \limsup_{\e\todown0}\mathscr C_\e(\bar\mu_\e) \leq \overline{\mathscr C}(\mu^*) = \min_{\M^1(\overline\Omega)}\overline{\mathscr C}. \] Therefore all inequalities are equalities. In particular, $\bar\mu$ is a minimizer of the limit problem and \[ \lim_{\e\todown0}\mathscr C_\e(\bar\mu_\e) = \min_{\M^1(\overline\Omega)}\overline{\mathscr C}. \] Moreover, \[ \lim_{\e\todown0} \inf_{\mu\in\M^1(\overline\Omega)} \mathscr C_\e(\mu) = \min_{\mu\in\M^1(\overline\Omega)} \overline{\mathscr C}(\mu). \] This proves Theorem~\ref{thm:conv-min}. \end{proof}

\bibliography{refs}
\bibliographystyle{plain}

\end{document}